\documentclass[11pt,letterpaper]{amsart}
\usepackage{graphicx}
\usepackage{amssymb,amscd,amsthm,amsxtra}
\usepackage{latexsym}
\usepackage{epsfig}
\usepackage{amsmath}
\usepackage{bm}
\usepackage{amssymb}
\usepackage{amsthm}
\usepackage[usenames,dvipsnames,svgnames,table]{xcolor}
\usepackage[linktocpage=true,colorlinks=true,linkcolor=Blue,citecolor=BrickRed,urlcolor=RoyalBlue]{hyperref}
\usepackage[alphabetic]{amsrefs}
\usepackage{comment}

\usepackage[colorinlistoftodos,prependcaption,textsize=tiny]{todonotes}
\usepackage{caption}

\usepackage{enumitem} 
\usepackage[mathscr]{euscript} 

\usepackage{mathrsfs} 

\usepackage{mathtools}

\numberwithin{equation}{section} 

\newtheorem{theorem}{Theorem}

\newtheorem{definition}{Definition}
\newtheorem{lemma}{Lemma}
\newtheorem{corollary}{Corollary}
\newtheorem{proposition}{Proposition}

\theoremstyle{remark}
\newtheorem{remark}{Remark}[section]

\newtheorem{Assumption}{Assumption}

\def \suchthat {\ \big | \ }

\title[Higher regularity to fully nonlinear elliptic equations]{Higher regularity of solutions to fully nonlinear elliptic equations}
\author[T. M. Nascimento]{Thialita M. Nascimento}
\address{Department of Mathematics, Iowa State University, 396 Carver Hall, 50011, Ames, IA, USA}{}
\email{thnasc@iastate.edu}
\author[G. Sá]{Ginaldo Sá}
\address{Department of Mathematics, University of Central Florida, 32816, Orlan\-do-FL, USA}{}
\email{ginaldo.sa@knights.ucf.edu}
\author[A. Sobral]{Aelson Sobral}
\address{Department of Mathematics, Universidade Federal da Para\'iba, 58059-900, Jo\~ao Pessoa-PB, Brazil}{}
\email{aelson.sobral@academico.ufpb.br}
\author[E. V. Teixeira]{Eduardo V. Teixeira}
\address{Department of Mathematics, University of Central Florida, 32816, Orlan\-do-FL, USA}{}
\email{eduardo.teixeira@ucf.edu}
\begin{document}
\maketitle

\begin{abstract} 
We establish higher regularity properties of solutions to fully nonlinear elliptic equations at interior critical points. The key novelty of our estimates lies in the fact that they yield smoothness properties that go beyond the inherent regularity limitations dictated by the heterogeneity of the problem. We explore various scenarios, revealing a plethora of improved regularity estimates. Notably, depending on the model's parameters, we establish estimates that transcend the natural regularity regime of the model, from $C^{0,\alpha_0}$ to $C^{1,\alpha_1}$ and further to $C^{2,\alpha_2}$, with the potential for even higher estimates. 
\tableofcontents
\end{abstract}

\section{Introduction}
We investigate higher regularity properties at critical points of viscosity solutions to uniform elliptic partial differential equations in the form
\begin{equation}\label{maineq}
F(x, D^2u) = f(x,u,Du).
\end{equation}
Caffarelli's celebrated a priori interior regularity theory, \cite{Caff}, establishes that viscosity solutions of fully nonlinear, uniform elliptic partial differential equations,
\begin{equation}\label{simpleEq}
    F(D^2u) = g(x),
\end{equation}
with source term $g\in L^p$, for $p>n$, are locally of class $C^{1,\delta}$ for an exponent $\delta>0$ depending only on dimension, ellipticity constants, and $p$. Notably, for non-homogeneous equations as outlined in \eqref{simpleEq}, Caffarelli's regularity theorem is, at its core, an optimal result;  see for instance \cites{CaffH, DKM, PT, PST, S} and references therein for regularity in other function spaces.

Regarding regularity properties of viscosity solutions to \eqref{maineq}, as long  minimal smoothness estimates are available to assure $f(x,u, Du) =: g(x)$ is an $L^p$ function for some $p>n$, Caffarelli's interior regularity theory applies. Classical bootstrap arguments may be used in the case when $x\mapsto f(x, \cdot \, , \cdot)$ has more regularity. Yet, if only $L^p$ bounds are accessible, the prospect of improving the regularity of solutions becomes, in principle, unattainable---at least within the scope of local regularity theory.

In certain problems, however, achieving sharp (possibly improved) control over the growth of solutions along special regions or points becomes imperative to advance the program. For instance, this is a core issue in the theory of free boundary problems as well as in certain geometric problems. These issues serve as fundamental motivations for the new results presented in this manuscript. 

Equations of the more general form as in \eqref{maineq} boast a rich historical legacy, with a plethora of applications. The theory varies considerably based on the hypotheses made on source term $f(x,u,Du)$. Prime models for which the new results proven in this article apply are 2nd-order differential equations of the Hamilton-Jacobi type, as the ones treated, for instance, in \cites{CDLP, CV, DSN, SN1, GN}. Similarly, singular elliptic problems, as the ones studied in \cites{BD1, BDL, BPRT, DFQ1, DFQ2, I}, can be rewritten as in \eqref{maineq}, and thus our results apply to those models too. 

Remarkably, our findings yield a gain of regularity exactly where the inherent characteristics of those models manifest, viz. at the corresponding singularities of the model. This regularity gain would probably be counterintuitive if understood purely from a PDE viewpoint. Similar phenomena, stemming from markedly different considerations, have been previously observed in variational problems, see for instance \cites{ATU1, ATU2, Teix4, Teix2, Teix5}. In turn, the key innovation within the results presented in this paper lies in promoting higher regularity precisely where limited information about the behavior of the solution is available, due to singularities.
 
Our quest of obtaining improved interior regularity estimates for viscosity solutions of \eqref{maineq} starts by showing that if the source term has a persisting singularity of order $-n/p$ at an interior point $x_0$, then the regularity of any viscosity solution of \eqref{simpleEq} can never surpass a given critical threshold; i.e. $C^{1,\alpha(n,p)}$ regularity at $x_0$, for a sharp, explicit exponent $\alpha(n,p)>0$, see Proposition \ref{Prop sharpness}. Interestingly, the sharp exponent $\alpha(n,p)$ is the same arising from the Potential Theory, via quite different considerations. In particular, it is impossible to improve regularity at such points. Drawing an analogy with free boundary problems, it is pertinent to interpret the result from Proposition \ref{Prop sharpness} as a non-degeneracy estimate.

Even more remarkably, we show that quantitative, higher regularity estimates are available at critical points or at vanishing points: 
\begin{equation}\label{alt-defC}
    \mathcal{C}(u) := u^{-1}(0) \cup |Du|^{-1}(0) =: \mathcal{C}_0(u) \cup \mathcal{C}_1(u).
\end{equation}
Even higher estimates can be obtained at vanishing critical points:
$$
    \mathcal{C}(u) := \{ u(x) = |Du(x)| = 0 \} = \mathcal{C}_0(u) \cap \mathcal{C}_1(u).
$$
In fact, in our theorems, to be properly stated in the next section, we allow independent decay of $f$ with respect to $u$ and to $Du$, and such flexibility yields meaningful gains in all possible scenarios. In essence, we establish a framework where if an interior point $x_0$ is a critical point but not a vanishing point, one can simply allow the rate of decay concerning $u$ to be zero. Consequently, all results will undergo appropriate adjustments in accordance with their respective theses. Similarly if one is investigating an interior point within the zero level set that is not a critical point, our results give sharp regularity information by letting the gradient decay rate go to zero.

We further investigate pointwise second-order differentiability of solutions. More precisely, we obtain sharp conditions on the source term $f$ under which viscosity solutions of \eqref{maineq} are actually $C^{2,\alpha}$ smooth at their interior critical points. Surprisingly, once solutions become twice differentiable, at inflection points, $\{D^2u(z) = 0\} =:\mathcal{C}_2(u)$, we obtain a gain of smoothness that surpasses the continuity of the medium. Specifically, we manage to show that $D^2u(x)$ exhibits geometric decay around inflection points, even in cases where the coefficients do not have geometric oscillation decay.

The literature on higher regularity estimates for solutions of fully nonlinear elliptic equations remains relatively sparse, featuring only a handful of foundational results. Notably, Savin's $C^{2,\alpha}$-regularity theorem, as outlined in \cite{Savin}, stands out. This theorem is specifically tailored for solutions to smooth (i.e. $F\in C^1$) fully nonlinear elliptic equations that are deemed ``flat'', i.e. with sufficiently small $L^\infty$ norm. Another important result in this thread of research pertains to  the pointwise second-order differentiaility of solutions to $C^1$--smooth fully nonlinear elliptic operators, up to a set of Hausdorff dimension $n-\epsilon$, \cite{ASS}. To a certain extent, the theorems established in this paper bear a philosophical kinship with these results. However, they diverge in their prerequisites, introducing a distinct dimension to the conditions leading to higher regularity within the analyzed context.

We conclude by noting that the insights guiding the results presented in this paper, at least in a heuristic sense, draw inspiration from techniques associated with free boundaries. It consists of interpreting critical points of viscosity solutions to \eqref{maineq} as if they were part of an {\it abstract} free boundary. 

Interestingly, the model investigated in this paper could, itself, be interpreted as a model for certain free boundary problems. In \cites{CS02, CSS04}, the authors investigated fully nonlinear elliptic equations arising from the theory of superconductivity, taking the form
$$
    F(x,D^2u) = g(x,u)\mathcal{X}_{\{|Du| \not = 0 \}}.
$$
Solutions are understood in a weak sense, where touching functions with gradient zero are disregarded. The main new insight there is to show that solutions satisfy ordinary viscosity inequalities and are thus entitled to the classical theory. The results proven in this paper, however, convey that in fact, the set of critical points somehow carries a richer regularity theory. Say, if the function $g(x,s)$ behaves like
$$
    |g(x,s)| \leq q(x)|s|^m,
$$
for $m>0$ and $q \in L^p(B_1)$, then our regularity result, Theorem \ref{improv reg cas p}, states that $u$ is $C^{1,\epsilon_{m,n,p}}$ at points in the set $\mathcal{C}(u)$, where
$$
    \epsilon_{m,n,p} \coloneqq \min \left\{ \frac{m+1-\frac{n}{p}}{(1 - m)_+},\, \alpha_* \right\}^-,
$$
and $\alpha_*$ represents the inherent theoretical limit for the gradient H\"older continuity of solutions to the homogeneous equation, a boundary set by the Nadirashvili--Vl\u{a}du\c{t} program,   \cites{NV1, NV2, NV3, NV4, NV5}; see definition \eqref{def alpha max}.

If $m \approx 1$, then we obtain that, along $\mathcal{C}(u)$, solutions are asymptotically as regular as $F$-harmonic functions. Furthermore, if there is enough structure in the diffusion operator $F$ and its coefficients, then solutions are entitled to Theorem \ref{l_infinity}, which provides $C^{2,\epsilon_1}$ regularity along $\mathcal{C}(u)$, where
$$
    \epsilon_1 \coloneqq \min \left\{2m-\frac{n}{p}, \tau, \beta_*  \right\}^-,
$$
and $\beta_*$ is the maximal Hessian H\"older continuity assured by the Evans' and Krylov's $C^{2,\beta_*}$ regularity theorem, \cites{E, K}; see also \cite{CC}*{Chapter 6}.



The rest of the paper is organized as follows: in Section \ref{sct Prelim} we provide preliminary definitions, establish the main structural assumptions we require in the PDE \eqref{maineq}, and state our main results. In Section \ref{sharpness estimates}, we make use of barriers to obtain the sharpness of the regularity results. In Section \ref{sct crit pt reg case p}, we establish improved $C^{1,\alpha}$ regularity results at critical points by means of a delicate asymptotic analysis. In Section \ref{sct crit pt reg case infty}, we establish improved $C^{2,\alpha}$ regularity results, and, provided Hessian degenerate points are regarded, we reach the sharp regularity exponent. In Section \ref{Sct Reg below n} we discuss higher regularity properties when the source integrability is below the dimension threshold. Transitioning to Section \ref{STC reg at extrema}, our focus shifts to exploring the regularity at local extrema points, without imposing any continuity assumptions on the coefficients. Notably, both in Section \ref{Sct Reg below n} and Section \ref{STC reg at extrema}, the natural regularity regime is merely $C^{0,\delta}$. However, our estimates transcend this baseline, yielding significantly higher regularity yielding higher differentiability properties of $u$ at those special points. Finally, in Appendix A we discuss Lipschitz estimates, and in  Appendix B we establish gradient growth estimates.

\section{Hypothesis and main results} \label{sct Prelim}

In this section, we present some preliminary definitions and assumptions on the structure of the above equation. We further state and discuss our main results.

\subsection{Preliminary definitions}

We consider diffusion problems in an open subset of the $n$-dimensional Euclidean space $\mathbb{R}^n$. Since our focus is on local and pointwise regularity results, we shall assume all equations are placed in the unit open ball $B_1$ with the center at the origin.

We denote by $\mbox{Sym}(n)$ the space of symmetric matrices of size $n\times n$ and, given constants $0 < \lambda \leq \Lambda$, we say that an operator $\mathcal{G} \colon \mbox{Sym}(n) \to \mathbb{R}$ is $(\lambda,\Lambda)$-elliptic if it satisfies
$$
    \mathcal{M}^-_{\lambda,\Lambda}(M-N) \leq \mathcal{G}(M) - \mathcal{G}(N) \leq \mathcal{M}^+_{\lambda,\Lambda}(M-N),
$$
for all $M, N \in \mbox{Sym}(n)$, where $\mathcal{M}^+_{\lambda,\Lambda}$ and $\mathcal{M}^-_{\lambda,\Lambda}$ stands for the {\it Pucci Extremal Operators} defined as
\begin{eqnarray*}
    \mathcal{M}^+_{\lambda,\Lambda}(M) &\coloneqq& \sup \left\{\mbox{Tr}\left(AM\right) \suchthat \mbox{spec}(A) \subseteq [\lambda,\Lambda]    \right\},\\
    \mathcal{M}^-_{\lambda,\Lambda}(M) &\coloneqq & \inf \left\{\mbox{Tr}(AM) \suchthat \mbox{spec}(A) \subseteq [\lambda,\Lambda] \right\},
\end{eqnarray*}
where $\mbox{spec}(A)$ denotes the set of eigenvalues of the matrix $A \in \mbox{Sym}(n)$.

\begin{definition}[Viscosity solution]
Let $\mathcal{G}\colon B_1 \times \mathbb{R} \times \mathbb{R}^n \times \mbox{Sym}(n) \to \mathbb{R}$ be a continuous function. We say that $u$ is a viscosity subsolution to
$$
    \mathcal{G}(x,u,Du,D^2u)=0 \quad \mbox{in} \quad B_1,
$$
if for every $x_0 \in B_1$ and  $\varphi \in C^2\left(B_r(x_0)\right)$, with $B_r(x_0) \Subset B_1$, such that
$$
    u \leq \varphi \quad \mbox{in} \quad B_r(x_0) \quad \mbox{and} \quad u(x_0) = \varphi(x_0),
$$
then
$$
    \mathcal{G}(x_0,\varphi(x_0), D\varphi(x_0), D^2\varphi(x_0)) \geq 0.
$$
We say that $u$ is a viscosity supersolution to
$$
    \mathcal{G}(x,u,Du,D^2u)=0 \quad \mbox{in} \quad B_1,
$$
if for every $x_0 \in B_1$ and  $\varphi \in C^2\left(B_r(x_0)\right)$, with $B_r(x_0) \Subset B_1$, such that
$$
    u \geq \varphi \quad \mbox{in} \quad B_r(x_0) \quad \mbox{and} \quad u(x_0) = \varphi(x_0),
$$
then
$$
    \mathcal{G}(x_0,\varphi(x_0), D\varphi(x_0), D^2\varphi(x_0)) \leq 0.
$$
A function is said to be a viscosity solution if it is both a sub and supersolution. 
\end{definition}

We indicate \cite{CIL92} for an account of the theory of viscosity solutions. It is worth noting that the results proven in \cite{Caff} (see also \cite{CC} for a more didactical account), as well as the ones presented here are understood as a priori estimates. We refer to \cite{CCKS} for a comprehensive theory of $L^p$-viscosity solutions.

Useful to the subsequent analysis, we define
\[
\mathcal{F}_{n, \lambda, \Lambda} \coloneqq  \left\{ u \in  C(\overline{B_1}) \ \middle\vert \begin{array}{c}
    F(D^2 u) = 0\, \text{in the viscosity sense in}\, B_1\, \mbox{for}  \\
    \mbox{some $(\lambda,\Lambda)$-elliptic operator}\, F\colon \mbox{Sym}(n) \to \mathbb{R}
  \end{array}\right\}.
\]
Although this may be a very large set of functions, it is known, see \cite{CC}, that there exists a universal modulus of continuity for the gradient of functions in $\mathcal{F}_{n, \lambda, \Lambda}$. More precisely, if $u \in \mathcal{F}_{n,\lambda,\Lambda}$, then there exists $C_*>0$ and $\alpha_*$ depending only on dimension and ellipticity constants such that
$$
    \|u\|_{C^{1,\alpha_*}(B_{{3}/{4}})} \leq C_* \|u\|_{L^\infty(B_1)}.
$$
Such an exponent $\alpha_*$ represents a theoretical barrier to the regularity theory of general viscosity solutions and can be defined as

\begin{equation}\label{def alpha max}
  \alpha_* \coloneqq \sup \left\{ \alpha \in (0,1) \ \middle\vert \begin{array}{c}
    \mbox{there exists} \, C_\alpha>0\, \mbox{such that} \, \\
    \|u\|_{C^{1,\alpha}(B_{3/4})} \leq C_\alpha \|u\|_{L^\infty(B_1)}, \forall u \in \mathcal{F}_{n,\lambda,\Lambda}
  \end{array}\right\}.
\end{equation}
Given an exponent $\alpha \in (0,1)$, whenever we write $\alpha^-$, we mean any number $0<\beta< \alpha$.

The analysis of this paper will be concentrated along the set of vanishing critical points, as described in the next definition.

\begin{definition}
For a function $v \in C^1$ we define
$$
    \mathcal{C}(v) = \left\{x \in B_1\suchthat v(x) = |Dv(x)| = 0 \right\},
$$
the set of points that are both zero and critical points of $v$.
\end{definition}

\subsection{Assumptions and results}

Let us discuss the structural assumptions on the operator in \eqref{maineq}. Throughout the paper, the operator $F$, responsible for the diffusion of the model, will be assumed to verify the following structural conditions:

\begin{Assumption}\label{unif ellipticity}
For every $x \in B_1$ fixed, the mapping
$$
    M \longmapsto F\left(x,M\right)
$$
is $(\lambda,\Lambda)$-elliptic.
\end{Assumption}

Monotonicity in the matrix variable is one of the key structural assumptions in order to make sense of the notion of viscosity solution.

The second assumption concerns the growth associated with the RHS of \eqref{maineq}.
\begin{Assumption}\label{convexity}
There exists $m\geq 0$ and $\gamma \geq 0$ such that the mapping $f\colon B_1 \times \mathbb{R} \times \mathbb{R}^n \to \mathbb{R}$ verifies
\begin{equation}\label{RHS_growth} 
    |f(x,s,\xi)| \le q(x)|s|^m \min\left\{1, |\xi|^\gamma \right\},
\end{equation}
where $q \in L^p(B_1)$ is a nonnegative function and $p>n$. 
\end{Assumption}
It is worth commenting that we will consider the term $\min \{1,|\xi|^\gamma \}$ in order to bypass \textit{a priori} Lipschitz estimates. See Appendix A for discussions on such an estimate.

The third assumption pertains to the oscillation of the coefficients of the operator $F$. To streamline our discussion, let us define the oscillation of these coefficients by:
$$
    \mbox{osc}_F(x,y) \coloneqq \sup\limits_{M \in \mbox{Sym}(n)} \frac{|F(x,M) - F(y,M)|}{\|M\| + 1} \quad \mbox{for} \quad x,y \in B_1.
$$

\begin{Assumption}\label{holder continuity of the coeffs}
There exist constants $\tau \in (0,1)$ and $C_\tau >0$ such that
$$
    \mbox{osc}_F(x,y) \leq C_\tau |x-y|^\tau \quad \mbox{for} \quad x,y \in B_1.
$$
\end{Assumption}

The H\"older continuity assumption on the coefficients of the diffusion operator is mainly to be used in Section \ref{growth est sec-prd version}, when improving regularity to the $C^2$ level. This is a natural assumption to attain such a level of regularity. As an $L^p$ theory is concerned, such an assumption can be weakened to some integrability condition of the oscillation function, see \cites{CC, Teix1} for further details. We will keep it as it is to ease the presentation of the results in the paper.

The last assumption concerns \textit{a priori} $C^{2,+}$ estimates:
\begin{Assumption}\label{a prior c2+ estimates}
Solutions to $F(0, D^2 h + N) = 0$ satisfy
\begin{equation}\label{inte C^{2,alp} estimates}
      \| h\|_{C^{2,\beta_*}(B_{1/2})} \le \Theta r^{-(2 + \bar{\epsilon})}\|h\|_{L^{\infty}(B_1)},
\end{equation}
for any $N \in \mbox{Sym}(n)$, with $F(0,N) = 0$.
\end{Assumption}

We now start discussing the statement of the new improved regularity estimates proven in this paper. Our first results concerns improved $C^{1,\alpha_*^-}$ regularity at critical points. This result offers a gain of smoothness, which is especially relevant when $p=n+\epsilon$, for some $0<\epsilon \ll 1$.

\begin{theorem}\label{improv reg cas p}
Let $u \in C (\overline{B_1} )$ be a normalized viscosity solution to 
\begin{equation}
    F(x, D^2u) = f\left(x,u,Du\right) \quad \text{in} \quad B_1. 
\end{equation}
Assume Assumptions \ref{unif ellipticity}, \ref{convexity} are in force and $F$ has a uniform continuous modulus of continuity in the coefficients. Then, $u$ is of class $C^{1,\epsilon_{m,\gamma,n,p}}$ at points in $\mathcal{C}(u)$, that is
\begin{equation}\label{estimate in thm 1}
    |u(x)| \le C|x-x_0|^{1 + \epsilon_{m,\gamma,n,p}},
\end{equation}
 for all $x \in B_{\frac{1}{4}}(x_0)$, where $C>0$ is a universal constant, $x_0 \in \mathcal{C}(u)$ and
$$
    \epsilon_{m,\gamma,n,p} \coloneqq \min \left\{ \frac{m+1-\frac{n}{p}}{(1 - (m+\gamma))_+},\, \alpha_* \right\}^-.
$$
\end{theorem}

Proceeding with the analysis, we provide regularity results at $C^{2,+}$ level, by requiring further, though natural, structural assumptions.

\begin{theorem}\label{l_infinity}
Let $u \in C(\overline{B_1} )$ is a normalized viscosity solution of  
\begin{equation}
    F(x, D^2u) = f(x,u,Du) \quad \text{in} \quad B_1.
\end{equation}
Assume Assumptions \ref{unif ellipticity}, \ref{holder continuity of the coeffs}, \ref{a prior c2+ estimates} are in force and \ref{convexity} holds with
\begin{equation}
    p > \frac{n(m+\gamma+1)}{2m+\gamma}.
\end{equation}
Given $x_0 \in \mathcal{C}(u)$, there exists a matrix $M_{x_0} \in \mbox{Sym}(n)$ such that 
\begin{equation}\label{estimate for all t}
 \left|u(x) - \frac{1}{2}M_{x_0}(x-x_0) \cdot (x-x_0)\right| \le C |x-x_0|^{2 + \epsilon_1},
\end{equation}
for all $x \in B_{1/4}(x_0)$, where
$$
    \epsilon_1 \coloneqq \min \left\{2m+\gamma-\frac{n}{p}, \tau, \beta_*  \right\}^-,
$$
and $C > 0$ is a universal constant.
\end{theorem}

It is worth noting that the regularity of the coefficients of the diffusion operator acts as a barrier to the regularity estimate, which is a natural phenomenon to expect; that is precisely why Assumption \ref{holder continuity of the coeffs} is critical to attaining $C^2$ regularity of solutions. 

Nonetheless, in the case where 
$$
    \mathcal{C}(u) \subset \left\{D^2u = 0 \right\},
$$
the previous theorem can be improved with fewer assumptions on the diffusion operator and a significant improvement on the regularity exponent. With this perspective, we have the following:

\begin{theorem}\label{Lp p large}
Let $u \in C (\overline{B_1} )$ is a normalized viscosity solution of  
\begin{equation}
    F(x, D^2u) = f(x,u,Du) \quad \text{in} \quad B_1,
\end{equation}
under Assumptions \ref{unif ellipticity}, \ref{a prior c2+ estimates}. Assume $F$ has a modulus of continuity of Dini type in the coefficients and \ref{convexity} holds with
\begin{equation}\label{restriction for p}
    p > \frac{n(m+\gamma+1)}{2m+\gamma}.
\end{equation}
Let $x_0 \in \mathcal{C}(u) \cap \{D^2u = 0 \}$. Then,  
\begin{equation}\label{growth est in thm3}
     \left|u(x)\right| \le C |x-x_0|^{2 + \epsilon_{m,\gamma,n,p}},
\end{equation}
for all $x\in B_{1/4}(x_0)$, where
$$
    \epsilon_{m,\gamma,n,p} \coloneqq \min \left\{\frac{2m + \gamma - \frac{n}{p}}{(1 - (m+\gamma))_+} ,\beta_* \right\}^-,
$$
and $C > 0$ is a universal constant. 
\end{theorem}
A natural extension of our analysis is when the integrability of the source term lies in $L^p$ with $p<n$. In \cite{ESC}, Escauriaza established the existence of a universal constant $\varepsilon_E \in (0,\frac{n}{2}]$, depending on dimension and ellipticity, such that solutions of $F(x,D^2u) = f \in L^{n-\nu}$, for $\nu < \varepsilon_E$, are entitled to the Harnack inequality, and thus are H\"older continuous for the sharp exponent $\alpha = \frac{n-2\nu}{n-\nu}$, according to \cite{Teix1}. We present the counterpart of our regularity results in this scenario. The proofs unfold through a parallel analysis, akin to the methodology employed in previously established theorems.

\begin{theorem}\label{reg below dimension}
Let $u \in C (\overline{B_1} )$ be a normalized viscosity solution to 
\begin{equation}
    F(x, D^2u) = f\left(x,u,Du\right) \quad \text{in} \quad B_1. 
\end{equation}
Assume Assumptions \ref{unif ellipticity}, \ref{convexity} are in force with $p = n-\nu$, where $\nu \in (0,\varepsilon_E)$, and $F$ has a uniform continuous modulus of continuity in the coefficients. Then, $u$ is of class $C^{0,\epsilon_{m,n,\nu}}$ at points in $\mathcal{C}_0(u)$, that is
\begin{equation}
    |u(x)| \le C|x-x_0|^{\epsilon_{m,n,\nu}},
\end{equation}
for all $x \in B_{\frac{1}{4}}(x_0)$, where $C>0$ is a universal constant, $x_0 \in \mathcal{C}_0(u)$ and
$$
    \epsilon_{m,n,\nu} \coloneqq \min \left\{ \frac{\frac{n-2\nu}{n-\nu}}{(1 - m)_+},\, 1 \right\}^-.
$$   
\end{theorem}
When there is an interplay between the amount of integrability of the RHS and the decay of the zeroth term, we have the following
\begin{theorem}\label{reg below dimension interplay}
Let $u \in C (\overline{B_1} )$ be a normalized viscosity solution to 
\begin{equation}
    F(x, D^2u) = f\left(x,u,Du\right) \quad \text{in} \quad B_1. 
\end{equation}
Assume Assumptions \ref{unif ellipticity}, \ref{convexity} are in force with $p = n-\nu$, where $\nu \in (0,\varepsilon_E)$, and $F$ has a uniform continuous modulus of continuity in the coefficients. Assume further that
\begin{equation}\label{interplay below dim}
    \frac{m\, n}{2m + 1} > \nu.
\end{equation}
Then, $u$ is differentiable at $x_0 \in \mathcal{C}_0(u)$ and there exists a universal constant $C>0$ such that
\begin{equation}
    |u(x) - Du(x_0)\cdot (x-x_0)| \le C|x-x_0|^{1+\epsilon_{m,n,\nu}},
\end{equation}
for all $x \in B_{\frac{1}{4}}(x_0)$, where
$$
    \epsilon_{m,n,\nu} \coloneqq \min \left\{ \frac{m\,n - \nu (m+1)}{n-\nu},\, \alpha_* \right\}^-.
$$   
\end{theorem}
Notice that if the interplay between integrability and decay on the RHS is stronger, we have the following corollary.
\begin{corollary}
Let $u \in C(\overline{B_1} )$ is a normalized viscosity solution of  
\begin{equation}
    F(x, D^2u) = f(x,u,Du) \quad \text{in} \quad B_1.
\end{equation}
Assume Assumptions \ref{unif ellipticity}, \ref{holder continuity of the coeffs}, \ref{a prior c2+ estimates} are in force and \ref{convexity} holds with $p = n-\nu$ and
\begin{equation}\label{stronger interplay below dim}
    \frac{n(m-1)}{2m} > \nu
\end{equation}
If $x_0 \in \mathcal{C}_0(u)$, then $u$ is twice differentiable at $x_0$ and
\begin{equation}\nonumber
 \left|u(x) - Du(x_0)\cdot (x-x_0)-\frac{1}{2}D^2u(x_0)(x-x_0) \cdot (x-x_0)\right| \le C |x-x_0|^{2 + \epsilon_1},
\end{equation}
for all $x \in B_{1/4}(x_0)$, where
$$
    \epsilon_1 \coloneqq \min \left\{2m-\frac{n}{n-\nu}, \tau, \beta_*  \right\}^-,
$$
and $C > 0$ is a universal constant.
\end{corollary}
We remark that if \eqref{stronger interplay below dim} holds, then \eqref{interplay below dim} is also true.

It is worth observing, as pointed out in the introduction, that our results are designed to allow independent decay of the RHS of \eqref{maineq}. To encapsulate the foregoing, we bring a few corollaries to elucidate. First and foremost, we present the consequences of Theorem \ref{improv reg cas p}.

\begin{corollary}
Let $u \in C(\overline{B}_1)$ be as in Theorem \ref{improv reg cas p}. If $x_0 \in \mathcal{C}_0(u) \backslash \mathcal{C}_1(u)$, then \eqref{estimate in thm 1} holds with
$$
    \epsilon_{m,n,p} \coloneqq \min \left\{ \frac{m+1-\frac{n}{p}}{(1 - m)_+},\, \alpha_* \right\}^-.
$$
If $x_0 \in \mathcal{C}_1(u) \backslash \mathcal{C}_0(u)$, then \eqref{estimate in thm 1} holds with
$$
    \epsilon_{\gamma,n,p} \coloneqq \min \left\{ \frac{1-\frac{n}{p}}{(1 - \gamma)_+},\, \alpha_* \right\}^-.
$$
\end{corollary}

It is interesting to observe that the decay from the zeroth order term provides, as expected, a higher regularity improvement.

We also provided its second-order version, a consequence of Theorem \ref{l_infinity}.

\begin{corollary}
Let $u \in C(\overline{B}_1)$ be as in Theorem \ref{l_infinity}. If $x_0 \in \mathcal{C}_0(u) \backslash \mathcal{C}_1(u)$, then \eqref{estimate for all t} holds with
$$
    \epsilon_1 \coloneqq \min \left\{2m-\frac{n}{p}, \tau, \beta_*  \right\}^-,
$$
If $x_0 \in \mathcal{C}_1(u) \backslash \mathcal{C}_0(u)$, then \eqref{estimate for all t} holds with
$$
    \epsilon_1 \coloneqq \min \left\{\gamma-\frac{n}{p}, \tau, \beta_*  \right\}^-.
$$
\end{corollary}

Finally, when Hessian degenerate points are concerned, we have the following consequence of Theorem \ref{Lp p large}.

\begin{corollary}
Let $u \in C(\overline{B}_1)$ be as in Theorem \ref{Lp p large}. If $x_0 \in \mathcal{C}_0(u) \cap \mathcal{C}_2(u) \backslash \mathcal{C}_1(u)$, then \eqref{growth est in thm3} holds with
$$
    \epsilon_{m,n,p} \coloneqq \min \left\{\frac{2m - \frac{n}{p}}{(1 - m)_+} ,\beta_* \right\}^-.
$$
If $x_0 \in \mathcal{C}_1(u) \cap \mathcal{C}_2(u) \backslash \mathcal{C}_0(u)$, then \eqref{growth est in thm3} holds with
$$
    \epsilon_{\gamma,n,p} \coloneqq \min \left\{\frac{\gamma - \frac{n}{p}}{(1 - \gamma)_+} ,\beta_* \right\}^-.
$$
\end{corollary}

We emphasize that these corollaries are actually {\it scholia} of the preceding theorems. In other words, they follow through amendments in the proofs of the main theorems rather than emerging as direct consequences of their respective theses. A key distinction lies in the construction of the approximating scheme, requiring minor adaptations, which are omitted here.

\subsection{Scaling properties}

We finish this section by discussing the scaling properties of equations of the form \eqref{maineq}. Let us assume $u$ solves
$$
    F(x, D^2u) = f(x,u,Du) \quad \mbox{in} \quad B_1,
$$
in the viscosity sense. Let $A$ and $B$ be positive constants and define
$$
    v(x) = \frac{u(Ax)}{B}.
$$ 
Direct computations show that $v$ is a viscosity solution to

$$
    \mathcal{F}(y,D^2v) = \overline{f}(y,v,Dv),
$$
where
$$
    \mathcal{F}(y,M) = \frac{A^2}{B} F\left (Ay, \frac{B}{A^2}M \right ),
$$
and the scaled font is given as,
$$
    \overline{f}(y,s,\xi) = A^{2}B^{-1}f\left(Ay,Bs,A^{-1}B\xi\right).
$$
Easily one cheks that the new operator $\mathcal{F}$ is $(\lambda,\Lambda)$-elliptic and
\begin{eqnarray*}
    |\overline{f}\left(y,s,\xi\right)| & = & A^{2}B^{-1}|f\left(Ay,Bs,A^{-1}B\xi\right)|\\
                                       & \leq & A^2 B^{-1}q(Ay)|Bs|^m \min \left\{1, |A^{-1}B \xi|^\gamma \right\}\\
                                       & = & A^{2-\gamma} B^{m+\gamma-1}|s|^m \min\left\{A^\gamma B^{-\gamma}, |\xi|^\gamma \right\}.  
\end{eqnarray*}

Picking $B := \max\{ \|u\|_{\infty}, 1\}$, we can assume, with no loss of generality,  that solutions are normalized, that is, $\|u \|_\infty \leq 1$.

\section{Sharpness}\label{sharpness estimates}

In this short session, we discuss the sharpness of Caffarelli's estimates, in the context of the main equation \eqref{maineq}, if no further structural conditions are taken into consideration. More precisely, we show that if the source function has a persistent singularity of order $-n/p$, at an interior critical point $x_0$, then the regularity of viscosity solutions is limited by the estimates arising from the Potential Theory. This is the contents of the following Proposition.

\begin{proposition}\label{Prop sharpness}
Let $u \in C (\overline{B}_1)$ be a viscosity solution to \eqref{maineq}. Assume further that Assumptions \ref{unif ellipticity} and \ref{convexity} are in force. If $x_0$ is an interior point and assume 
$$
    \inf_{B_r(x_0)} f \geq \delta r^{-\frac{n}{p}},
$$
for all $0 < r < r_0$. Then
\begin{equation}\label{sharpness with rhs blowing up}
    \limsup_{x \to x_0}\left( \frac{u(x)-u(x_0)}{|x-x_0|^{2 - \frac{n}{p}}}\right) \geq \frac{\delta}{\Lambda \left(2-\frac{n}{p} \right)\left(n + \frac{n}{p} \right)}.
\end{equation}
In particular, if $x_0$ is a critical point, then $u$ fails to be $C^{2 - \frac{n}{p} +\epsilon}$ at $x_0$, for all $\epsilon>0$.
\end{proposition}
\begin{proof}
Let us define
$$
    w(x) \coloneqq C|x-x_0|^{2 - \frac{n}{p}} + u(x_0),
$$
where
$$
    C \coloneqq \frac{\delta}{\Lambda \left(2-\frac{n}{p} \right)\left(n + \frac{n}{p} \right)}.
$$
Direct computations show that
$$
    D^2w(x) = \frac{C\left(2- \frac{n}{p}\right)}{|x-x_0|^{\frac{n}{p}}} \left(I_n + \frac{n}{p} \frac{x-x_0}{|x-x_0|} \otimes \frac{x-x_0}{|x-x_0|} \right),
$$
and so
$$
    F(x,D^2w(x)) \leq \mathcal{M}^+_{\lambda,\Lambda} (D^2w(x)) = \Lambda \left(2-\frac{n}{p} \right)\left(n + \frac{n}{p}\right)C|x-x_0|^{-\frac{n}{p}}.
$$
Thus, if $x \in \partial B_r(x_0)$ and by the choice of the constant $C$, it holds that
$$
    F(x,D^2w) \leq \delta r^{-\frac{n}{p}} \leq F(x,D^2u).
$$
As a consequence, since $w(x_0) = u(x_0)$, for each $r>0$, there must exist a $x_r \in \partial B_r(x_0)$ such that
$$
    w(x_r) \leq u(x_r),
$$
from which \eqref{sharpness with rhs blowing up} follows.
\end{proof}

In the superquadratic regime, i.e. when $\gamma>2$ in Assumption \ref{convexity}, while our theorems still yield improved regularity at critical points, locally solutions are, in general, no better than H\"older continuous, see for instance \cite{CV}. 

By a slight adaptation of the previous barrier argument, we obtain a quantitative upper bound for the optimal H\"older continuity exponent of solutions in the superquadratic regime. 

\begin{proposition}\label{Prop sharpness 2}
Assume $n\geq 2$ and let $u \in C(\overline{B}_1)$ be a viscosity solution to \eqref{maineq}. Assume further that Assumptions \ref{unif ellipticity} and \ref{convexity} are in force with $\gamma>2$ and 
$$
    \inf_{B_r(x_0)} f \geq \delta r^{-\frac{\gamma}{\gamma-1}},
$$
for all $0 < r < r_0$. Then
\begin{equation}\label{sharpness with rhs blowing up 2}
    \limsup_{x \to x_0}\left( \frac{u(x)-u(x_0)}{|x-x_0|^{\frac{\gamma-2}{\gamma-1}}}\right) \geq \frac{\delta}{\Lambda \left(\frac{\gamma-2}{\gamma-1} \right)\left(n - \frac{\gamma}{\gamma-1}\right) }.
\end{equation}
In particular, $u$ fails to be $C^{\frac{\gamma-2}{\gamma-1} +\epsilon}$ at $x_0$, for all $\epsilon>0$. 
\end{proposition}
\begin{proof}
Let us define
$$
    w(x) \coloneqq C|x-x_0|^{\frac{\gamma-2}{\gamma-1}} + u(x_0),
$$
where
$$
    C \coloneqq \frac{\delta}{\Lambda \left(\frac{\gamma-2}{\gamma-1} \right)\left(n - \frac{\gamma}{\gamma-1}\right)}.
$$
Observe that since $n\geq 2$ and $\gamma>2$, we have $n - \frac{\gamma}{\gamma-1}>0$. Direct computations show that
\begin{eqnarray*}
    Dw(x) &=& C \left(\frac{\gamma-2}{\gamma-1} \right)|x-x_0|^{\frac{-\gamma}{\gamma-1}}\\
    D^2w(x) & = &  C \left(\frac{\gamma-2}{\gamma-1} \right)|x-x_0|^{\frac{-\gamma}{\gamma-1}}\left(I_n - \left(\frac{\gamma}{\gamma-1}\right)\frac{x-x_0}{|x-x_0|} \otimes \frac{x-x_0}{|x-x_0|} \right)
\end{eqnarray*}
and so
$$
    F(x,D^2w(x)) \leq \mathcal{M}^+_{\lambda,\Lambda} (D^2w(x)) = \Lambda \left(\frac{\gamma-2}{\gamma-1} \right)\left(n - \frac{\gamma}{\gamma-1}\right)C|x-x_0|^{-\frac{\gamma}{\gamma-1}}.
$$
Thus, if $x \in \partial B_r(x_0)$ and by the choice of the constant $C$, it holds that
$$
    F(x,D^2w) \leq \delta r^{-\frac{\gamma}{\gamma-1}} \leq F(x,D^2u).
$$
As a consequence, since $w(x_0) = u(x_0)$, for each $r>0$, there must exist a $x_r \in \partial B_r(x_0)$ such that
$$
    w(x_r) \leq u(x_r),
$$
from which \eqref{sharpness with rhs blowing up} follows.
\end{proof}

\section{$C^{1,\alpha}$ regularity improvement} \label{sct crit pt reg case p}

This section is dedicated to the proof Theorem \ref{improv reg cas p}. The starting point of the proof is the (already known) Caffarelli's $C^{1,\alpha_p}$ regularity estimate. If $u$ is a normalized viscosity solution to
$$
    F(x, D^2u) = f(x,u,Du),
$$
then, Assumption \ref{convexity} assures that the RHS is an $L^p$ function for $p > n$. Therefore, it falls into the scope of \cite{Caff}, see also \cite{Teix1} for optimality, for which it holds that
$$
    \alpha_p = \min \left\{1 - \frac{n}{p}, \alpha_*^- \right\},
$$
where $\alpha_*$ is the universal exponent associated to functions in $\mathcal{F}_{n,\lambda,\Lambda}$, defined in \eqref{def alpha max}.

\subsection{Gain of regularity}

As mentioned before, $C^{1,\alpha_p}$ regularity estimates are automatically true, for
$$
    \alpha_p = \min\left\{1 - \frac{n}{p}, \alpha_*^- \right\},
$$
where $\alpha_*$ is the associated exponent to the regularity theory for the homogeneous equation with constant coefficients.

In what follows we will use the following notation: 
$$
    \tilde{F}_{\mu}(x, M): = \mu^2 F(\mu x, \mu^{-2} M).
$$

Recall, from assumption \ref{convexity}, the RHS satisfies
$$
    |f(x, s, \xi)| \leq q(x)|s|^m \min\left\{1,|\xi|^{\gamma}\right\},
$$
for some nonnegative function $q(x) \in L^p(B_1)$, for $p > n$, and $m, \gamma \ge 0$.

\begin{lemma}[Approximation lemma]\label{Approx lemma case L^p}
Let $u \in C (\overline{B}_1 )$ be a normalized viscosity solution of 
\begin{equation}\label{main eq case p}
\tilde{F}_{\mu}(x, D^2u) = f(x, u, Du) \quad \text{in} \quad B_1.
\end{equation}
Assume $0 \in \mathcal{C}(u)$. Given $\delta > 0$ there exists $\epsilon = \epsilon(\delta, n, \lambda, \Lambda)$ such that if 
$$
    \|f (x,u(x),Du(x))\|_{L^{p}(B_1)} < \epsilon \quad \text{and} \quad \mu < \epsilon,
$$
then there exists $h \in \mathcal{F}_{n,\lambda,\Lambda}$, such that $0 \in \mathcal{C}(h)$ and 
$$
    \|u - h\|_{L^{\infty}(B_{1/2})} < \delta. 
$$
\end{lemma}

\begin{proof}
Assume, seeking a contradiction, that for some $\delta_0 > 0$, there exists a sequence $(u_k, f_k, \mu_k)_{k \in \mathbb{N}} \subset C (\overline{B}_1 ) \times L^p(B_1) \times \mathbb{R}^+$ satisfying
\begin{enumerate}[label=(\roman*)]

    \item $u_k$ is normalized;

    \item $0 \in \mathcal{C}(u_k)$;

    \item $\displaystyle \max \left\{ \|f_k\|_{L^p(B_1)}, \mu_k \right\} \leq \frac{1}{k} $;

    \item $\tilde{F}_{\mu_k} (x, D^2 u_k ) = f_k(x, u_k, D u_k)$ in $B_1$;
\end{enumerate}
however, 
\begin{equation}\label{Appr Lem contrad eq case L^p 0}
    \mbox{dist}\left[u_k, \mathcal{F}_{n,\lambda, \Lambda}\right] \geq \delta_0, 
\end{equation} 
for all $k\ge 1$. By our assumptions on $f$ and the diffusion operator, we have $\{ u_k \}_{k \in \mathbb{N} } \in C_{loc}^{1,\alpha_p} (B_1)$, with universal estimates. Therefore, passing to a subsequence if necessary, we obtain
$$
    (u_k, Du_k) \to 
        (u_\infty, Du_\infty)
$$
locally uniform in $L^\infty(B_1) \times L^\infty(B_1)$; in particular we deduce that $0 \in \mathcal{C}(u_\infty)$. Moreover, through a further subsequence in necessary, we obtain $\tilde{F}_{\mu_k} \to F'$ locally uniformly on $B_1 \times \mbox{Sym}(n)$ and  $f_k \to 0$. Thus, by stability results in the theory of viscosity solutions, we have
$$
    F'(D^2 u_{\infty}) = 0 \quad \text{in} \quad B_{3/4},
$$
for some $(\lambda,\Lambda)-$elliptic operator $F'$, which contradicts \eqref{Appr Lem contrad eq case L^p 0} for $k$ sufficiently large. 
\end{proof}

Assuming $ 0 \in \mathcal{C}(u)$ and that $u$ is  normalized $C_{loc}^{1,\alpha_p}$-solution of \eqref{main eq case p} we have, in particular, that for all $ 0 < t \le 1/2$, 
\begin{equation}
    \sup\limits_{B_t} |Du | \le C  t^{\alpha_p}. 
\end{equation}
Set $\alpha_p < \epsilon_1 < \alpha_{*}$ as 
\begin{equation}\label{new holdet exponente case p}
     \epsilon_1 \coloneqq \min \left\{ \frac{(m+1 -\frac{n}{p}) + (m+\gamma)\alpha_p}{1+\theta}, \alpha_{*}^{-} \right\},
\end{equation}
for a $\theta>0$ to be chosen later. We emphasize that this special choice for $\theta$ will be important in the asymptotic analysis.

We are now ready to prove the  $C^{1,\epsilon_1}$ regularity of $u$ at the origin. 

\begin{proposition}\label{bootstrap step 1}
Let $u \in C (\overline{B}_1  )$ be as in Theorem \ref{improv reg cas p}. If
$$
    \sup_{B_t(x_0)} |Du| \leq C_0t^{\alpha_p},
$$
then
$$
    \sup_{B_t(x_0)} |Du| \leq C_1t^{\epsilon_1},
$$
for any $x_0 \in \mathcal{C}(u)$, where $\epsilon_1$ is as defined in \eqref{new holdet exponente case p}.
\end{proposition}
\begin{proof}
We assume $x_0 = 0$. The general case is followed by a translation. For $ 0 < \rho < 1/2$ to be chosen later, define
$$
    v(x) \coloneqq u(\rho x) \quad x \in B_1.
$$
It is easily checked that $v$ satisfies
$$
    F_{\rho^2}(x,D^2v) = f_{\rho}(x, v, D v), \quad \text{in}  \quad  B_1
$$
where $F_{\rho^2} (x, M) = \rho^2 F(\rho x,\rho^{-2}M)$ and $f_{\rho}(x,s,\xi)=\rho^2 f\left(\rho x,s,\rho^{-1}\xi\right)$. Moreover, note that
$$
    |f_{\rho}(x, v, D v)|= \rho^2\left|f(\rho x, v,\rho^{-1} Dv)\right|\leq \rho^{2}q(\rho x) |u(\rho x)|^m |D u(\rho x)|^{\gamma}.
$$
Since $0 \in \mathcal{C}(u)$, and $u \in C_{loc}^{1,\alpha_p}$, then in particular
\begin{equation}\label{grad local reg}
    \sup\limits_{x \in B_{\rho}} \left\{|u(x)|, \rho|Du(x)| \right\} \le C' \rho^{1+\alpha_p}, 
\end{equation}
for a universal $C' > 0$. Therefore,
\begin{eqnarray*}
    \|f_{\rho}(x, v, D v)\|_{p} & \le & C'^{(m+\gamma)} \rho^{2 + m(1+\alpha_p) + \alpha_p\gamma} \|q(\rho-)\|_{L^p(B_1)}\\
                                & \le & C'^{(m+\gamma)} \rho^{2 + m(1+\alpha_p) + \alpha_p\gamma - \frac{n}{p}} \|q\|_{L^p(B_1)}.
\end{eqnarray*}
Recall that $v$ is a normalized solution, and if $\rho>0$ is small enough, we can apply Lemma \ref{Approx lemma case L^p} in order to find $h \in \mathcal{F}_{n,\lambda,\Lambda}$ such that 
$$
    \| v - h \|_{L^{\infty}(B_{1/2})} < \delta, 
$$ 
for some $\delta>0$ to be chosen later. In view of the $C_{loc}^{1,\alpha_*}$ interior regularity of $h$ and since $0 \in \mathcal{C}(h)$, we have 
\begin{eqnarray}\label{estimate for u}
    \sup\limits_{B_{\rho^2}} | u(x) | &=&  \sup\limits_{B_{\rho}} | v(x)  | \nonumber \\
                                    &\le&  \sup\limits_{B_{\rho}} | v(x) - h(x) | +  \sup\limits_{B_{\rho}} \left| h(x)\right| \nonumber \\
                                    &\le& \delta + C^*\rho^{1+\alpha_*} \nonumber \\
                                    &=& \delta + C^* \rho^{\alpha_* - \epsilon_1} \rho^{1 + \epsilon_1} \nonumber \\
                                    &\le& \delta + \frac{1}{2} \rho^{1 + \epsilon_1},
\end{eqnarray}
for a $\rho > 0$ so small that 
\begin{equation}\label{choice of rho}
    \max\left\{C'^{(m+\gamma)} \rho^{1+\epsilon_1 + m(1+\alpha_p) + \alpha_p\gamma - \frac{n}{p}} \|q\|_{L^p(B_1)}, \,  2C^* \rho^{\alpha_*-\epsilon_1} \epsilon^{-1},\,    \rho^2  \right\} < \epsilon, 
\end{equation}
where $\epsilon > 0 $ is the correspondent smallness regime of Lemma \ref{Approx lemma case L^p} with $\delta$ taken to be $$
    \delta := \frac{1}{2} \rho^{1 + \epsilon_1}.
$$
In conclusion, we have established
\begin{equation}\label{iteration k=1 case p}
    \sup\limits_{B_{\rho^2}} | u(x) | \le \rho^{1 + \epsilon_1}. 
\end{equation}
Next, by means of scaling analysis, we want to show that   
\begin{equation}\label{iteration  case p}
    \sup\limits_{B_{\rho^{k+1}}} \left| u(x)  \right| \leq \rho^{k(1 + \epsilon_1)} 
\end{equation}
holds for all $k \ge 1$. This is achieved through induction. The case $k = 1$ is precisely the estimate in \eqref{iteration k=1 case p}.   Now, for the induction step, we assume that \eqref{iteration  case p} is verified for $1, \dots, k$, and  let $v_k \colon B_{1} \to \mathbb{R}$ be defined as
$$
    v_k(x) : = \frac{u(\rho^{k+1} x) }{\rho^{k(1 + \epsilon_1)}} .
$$
Thus, by induction hypothesis, $v_k$ is a normalized solution to 
$$
    F_k(x,D^2 v_k) = f_k(x, v_k, Dv_k),
$$
where
\begin{eqnarray*}
    F_k(z,M) &=& \rho^{2 + k(1- \epsilon_1)}F(\rho^{k+1} z, \rho^{- (2 + k(1 - \epsilon_1))}M)\\
    f_k(z,s,\xi) &=& \rho^{2 + k(1-\epsilon_1)}, f\left(\rho^{k+1} x,\rho^{-k(1+\epsilon_1)} s, \rho ^{k\epsilon_1 - 1} \xi \right). 
\end{eqnarray*}
Observe that $F_k$ is a $(\lambda,\Lambda)$-elliptic operator and  
$$
    f_k(z,s,\xi) = \rho^{2 + k(1-\epsilon_1)} f\left(\rho^{k+1} x,\rho^{-k(1+\epsilon_1)} s, \rho ^{k\epsilon_1 - 1} \xi \right).
$$
Since $0 \in \mathcal{C}(u)$, estimate \eqref{grad local reg} leads to
\begin{eqnarray*}
    |f_k(x,v_k,Dv_k) | &\le& \rho^{2 + k(1-\epsilon_1)} q(\rho^{k+1} x)\left|u\left(\rho^{k+1}x\right)\right|^m\, \left|Du\left(\rho^{k+1} x\right)\right|^{\gamma} \nonumber \\
    &\leq  & C'^{(m+\gamma)} \rho^{{2 + k(1-\epsilon_1)} + m(1+\alpha_p)(k+1) + \gamma \alpha_p (k+1)} q(\rho^{k+1} x ),
\end{eqnarray*}
and so,
\begin{eqnarray}\label{smallness step 1}
    \|f_k\|_{L^p(B_1)} & \le & C'^{(m+\gamma)} \rho^{{2 + k(1-\epsilon_1)} + m(1+\alpha_p)(k+1) + \gamma \alpha_p (k+1) - \frac{n}{p}(k+1)} \| q \|_{L^p(B_1)} \nonumber \\ 
    & = & C'^{(m+\gamma)} \rho^{{1 + \epsilon_1 + (k+1)\left(m(1+\alpha_p) + \gamma \alpha_p  - \frac{n}{p} + 1 - \epsilon_1\right)}} \| q \|_{L^p(B_1)} \\
    & \leq & C'^{(m+\gamma)} \rho^{{1 + \epsilon_1 + (k+1)\frac{\theta\left(m(1+\alpha_p) + \gamma \alpha_p  - \frac{n}{p} + 1\right)}{1+\theta}}} \| q \|_{L^p(B_1)}. \nonumber
\end{eqnarray}
By \eqref{choice of rho}, the source term is in a smallness regime. Therefore, $v_k$ is entitled to Lemma \ref{Approx lemma case L^p} which, along with the choice made in  \eqref{choice of rho}, yields
$$
    \sup\limits_{B_{\rho}} | v_k(x)| \le \rho^{1 + \epsilon_1},
$$
proving therefore the induction thesis. Now, since \eqref{iteration  case p} holds for every $k \in \mathbb{N}$, given $t<1/2$, there exists $k_0 \in \mathbb{N}$ such that 
$$
    \rho^{k_0+1} \leq t \leq \rho^{k_0},
$$
and so, since $B_t \subseteq B_{\rho^{k_0}}$, it holds
\begin{eqnarray*}
    \sup_{x \in B_t}|u(x)| & \leq & \rho^{(k_0-1)(1+\epsilon_1)}\\
        & = & \left(\rho^{-2(1+\epsilon_1)} \right) \rho^{(k_0+1)(1+\epsilon_1)} \leq \rho^{-2(1+\epsilon_1)}t^{1+\epsilon_1}, 
\end{eqnarray*}
and the Proposition is proven by applying Lemma \ref{growth implies grad est}.
\end{proof}

Recall that the exponent $\epsilon_1 > \epsilon_0 \coloneqq \alpha_p$. The key remark now is that we can repeat the whole process above delineated, by using the newly achieved estimate,
$$
    \sup\limits_{B_t} |Du(x)| \le C_1 t^{\epsilon_1}
$$
as a replacement of \eqref{smallness step 1}. A careful analysis yields 
$$
    \sup\limits_{B_t} |u(x)| \le C_2 t^{1 + \epsilon_2}
$$
for a $\epsilon_2 > \epsilon_1$ given by
\begin{equation}\label{exponent beta_2}
    \epsilon_2 = \min \left\{ \frac{(m+1- \frac{n}{p}) + (m+ \gamma)\epsilon_1}{1+\theta} , \alpha_{*} ^{-} \right\},
\end{equation}
where $\theta>0$ is to be precised later. This argument can be repeated indefinitely, which gives the following result:
\begin{proposition}\label{bootstrap exponent}
Let $u \in C( \overline{B}_1 )$ be as in Theorem \ref{improv reg cas p}. If
$$
    \sup_{B_t(x_0)} |Du(x)| \leq C_k t^{\epsilon_k},  
$$
then
$$
    \sup_{B_t(x_0)} |Du(x)| \leq C_{k+1} t^{\epsilon_{k+1}},  
$$
for any $x_0 \in \mathcal{C}(u)$, where
\begin{equation}\label{recursive exponent relation}
    \epsilon_{k+1} \coloneqq \min \left\{\frac{(m+1 - \frac{n}{p}) + (m+\gamma)\epsilon_k}{1+\theta}, \alpha_*^-  \right\}.
\end{equation}
\end{proposition}
\begin{proof}
We assume $x_0=0$. The argument is identical to the one from the proof of Proposition \ref{bootstrap step 1}, with \eqref{grad local reg} replaced by
\begin{equation}\nonumber
    \sup\limits_{x \in B_{\rho}} \left\{|u(x)|, \rho|Du(x)| \right\} \le C' \rho^{1+\epsilon_k}. 
\end{equation}   
\end{proof}
\subsection{Asymptotic analysis}

We have proved that if
$$
    \sup\limits_{x \in B_{\rho}} |Du(x)| \le C_0 \rho^{\alpha_p}, 
$$
then,
$$
    \sup\limits_{x \in B_{\rho}} |Du(x)|  \le C_1 \rho^{\epsilon_1}, 
$$
for a slightly greater exponent $\epsilon_1>\epsilon_0\coloneqq \alpha_p$ given by \eqref{new holdet exponente case p} and a (quantified) constant $C_1>0$. We can now repeat the entire  argument  scheme with the newly achieved estimate, as in Proposition \ref{bootstrap exponent}, in order to obtain the recursive sequence of exponents \eqref{recursive exponent relation}, for which we provide an asymptotic analysis.

\begin{proposition}\label{asymptotic analysis case 1,alfa}
Let $\{\epsilon_k \}_{k \in \mathbb{N}}$ be the nondecreasing recursive sequence defined as in \eqref{recursive exponent relation}. Then
$$
    \epsilon_{m,\gamma,n,p,\theta} \coloneqq \lim_{k \rightarrow \infty} \epsilon_k
$$
exists and
$$
    \epsilon_{m,\gamma,n,p,\theta} = \min \left\{ \frac{m+1-\frac{n}{p}}{(1 + \theta - (m+\gamma))_+},\, \alpha_*^- \right\}.
$$
\end{proposition}
\begin{proof}
First, we recall that from the construction, the sequence $\epsilon_k$ is so that
$$
    \epsilon_k \leq \epsilon_{k+1} \quad \mbox{for every} \quad k \in \mathbb{N}.
$$
Moreover, $0 \leq \epsilon_k \leq \alpha_*^-$, and so, being a bounded monotone sequence, $\lim_{k \rightarrow \infty} \epsilon_k$ exists. Let
$$
    k_0 \coloneqq \sup\left\{i \in \mathbb{N} \, \Big{\vert}\, \epsilon_{i+1} = \frac{\left(m+1 - \frac{n}{p}\right) + (m+\gamma)\epsilon_{i}}{1+\theta}\right\}.
$$
Since $\{\epsilon_k \}_{k \in \mathbb{N}}$ is nondecreasing, we get that
$$
    \epsilon_{k+1} = \frac{\left(m+1 - \frac{n}{p}\right) + (m+\gamma)\epsilon_{k}}{1+\theta},
$$
for every $k \leq k_0$. We claim that
$$
    \epsilon_{k_0+1} = \frac{\left(m+1-\frac{n}{p}\right)}{1+\theta} \sum_{l=0}^{k_0-1} \left(\frac{m+\gamma}{1+\theta} \right)^l + \alpha_p \left(\frac{m+\gamma}{1+\theta} \right)^{k_0}.
$$
If $k_0=1$, then it follows by \eqref{exponent beta_2}. We proceed by induction. Assume it holds up to $j$ and let us show it holds also to $j+1$. By Proposition \ref{bootstrap exponent}, it holds
$$
    \epsilon_{j+1} = \frac{\left(m+1 - \frac{n}{p}\right) + (m + \gamma)\epsilon_{j}}{1+\theta}.
$$
By induction assumption,
$$
    \epsilon_j = \frac{\left(m+1-\frac{n}{p}\right)}{1+\theta} \sum_{l=0}^{j-1} \left(\frac{m + \gamma}{1+\theta} \right)^l + \alpha_p \left(\frac{m + \gamma}{1+\theta} \right)^j.
$$
To simplify notation, define
$$
    \Pi := \frac{\left(m+1 - \frac{n}{p} \right)}{1+\theta}.
$$
Then
\begin{eqnarray*}
    \epsilon_{j+1} & = & \Pi + \frac{(m+ \gamma)}{(1+\theta)}\epsilon_j\\
                  & = & \Pi + \frac{(m+ \gamma)}{(1+\theta)}\left[\Pi \sum_{l=0}^{j-1} \left( \frac{m + \gamma}{1+\theta} \right)^l + \alpha_p \left(\frac{m + \gamma}{1+\theta} \right)^j \right]\\
                  & = & \Pi +  \left[\Pi \sum_{l=1}^{j} \left( \frac{m + \gamma}{1+\theta} \right)^l + \alpha_p \left(\frac{m + \gamma}{1+\theta} \right)^{j+1} \right].
\end{eqnarray*}
Therefore,
$$
    \epsilon_{j+1} = \frac{\left(m+1-\frac{n}{p}\right)}{1+\theta} \sum_{l=0}^{j} \left( \frac{m + \gamma}{1+\theta} \right)^l + \alpha_p \left(\frac{m + \gamma}{1+\theta} \right)^{j+1},
$$
from which follows the claim. If $k_0 = \infty$, then it follows that
$$
    \frac{m+\gamma}{1+\theta} < 1,
$$
otherwise, if that is not the case, then $\epsilon_{k_0+1} = \infty$, which is a contradiction, since $\epsilon_{k} < \infty$ for every $k \in \mathbb{N}$. But now we have a geometric series, and so
\begin{eqnarray*}
    \epsilon_{k_0 + 1} & = & \frac{m+1-\frac{n}{p}}{1+\theta} \sum_{l=0}^\infty \left(\frac{m + \gamma}{1+\theta} \right)^l\\
                          & = & \frac{m+1-\frac{n}{p}}{1+\theta} \left( \frac{1+\theta}{1+\theta - (m+\gamma)} \right) = \frac{m+1-\frac{n}{p}}{1+\theta - (m+\gamma)}.
\end{eqnarray*}
If $k_0 < \infty$, then, by definition of $k_0$, we have $\epsilon_{k_0 + 2} = \alpha_*^-$, and there is nothing further to be done.
\end{proof}

We finish this section by gathering all results in order to deliver the proof of Theorem \ref{improv reg cas p}.

\begin{proof}[Proof of Theorem \ref{improv reg cas p}]
First, we observe that it is enough to prove the case that $x_0 = 0 \in \mathcal{C}(u)$. The general case follows by a translation. Given $\theta>0$, we apply Proposition \ref{bootstrap step 1} and \ref{bootstrap exponent}, in order to obtain, inductively, a sequence $(\epsilon_{k,\theta}, C_{k,\theta})_{k \in \mathbb{N}}$ such that
$$
    \sup_{B_t} |Du(x)| \leq C_{k,\theta}\, t^{\epsilon_{k,\theta}}.
$$
The Theorem is proved once we notice, due to Proposition \ref{asymptotic analysis case 1,alfa}, that
$$
     \lim_{\theta \to 0}\lim_{k \to \infty} \epsilon_{k,\theta} = \min \left\{\frac{m+1 - \frac{n}{p}}{(1- (m+\gamma))_+}, \alpha_*^- \right\},
$$
and therefore, by continuity, it holds
$$
    \epsilon_{k,\theta} \geq \min \left\{\frac{m+1 - \frac{n}{p}}{(1- (m+\gamma))_+}, \alpha_* \right\}^-.
$$
\end{proof}

\section{$C^{2,\alpha}$ regularity improvement} \label{sct crit pt reg case infty}

In this section, we prove that viscosity solutions of 
\begin{equation}
    F(x,D^2u) = f(x,u,Du) \quad \text{in} \quad B_1,
\end{equation}
are of class $C^{2,\alpha}$, for some exponent $\alpha$ to be described, provided Assumptions \ref{unif ellipticity}, \ref{holder continuity of the coeffs} are in force and assumption \ref{convexity} holds for $p>n$ large enough.

\subsection{Hessian regularity at critical points}

We provide some useful notations to ease the presentation. Given a function $w \in L^\infty_{loc}(B_1)$, a subset $D \subset L^\infty_{loc}(B_1)$ and a ball $B \Subset B_1$ we define
$$
    {dist}_B[w,D] \coloneqq \inf_{v \in A} \|w - v\|_{L^\infty(B)}.
$$
Given a matrix $X$, define $P_X(y) = \frac{1}{2}Xy \cdot y$ and
$$
    \mathcal{F}^2_{n,\lambda,\Lambda} \coloneqq \left\{X \in \mbox{Sym}(n)\suchthat P_X \in \mathcal{F}_{n,\lambda,\Lambda}\right\}.
$$
Observe that $\mathcal{F}^2_{n,\lambda,\Lambda} \subset \mathcal{F}_{n,\lambda,\Lambda}$.

As usual, the first step is to prove an approximation lemma. The arguments follow the lines of Lemma \ref{Approx lemma case L^p}. We will prove it once more in a slightly different form in order to ease the subsequent iteration argument.

\begin{lemma}\label{Approx lemma}
Let $u \in C (\overline{B}_1 )$ be a normalized viscosity solution to
$$
    F(x,D^2u) = f(x,u,Du) \quad \mbox{in} \quad B_1. 
$$
Assume $0 \in \mathcal{C}(u)$ and $F(0,0) = 0$. Given $r_0>0$ and $\epsilon_0 < \beta_*$, there exists $\eta_0>0$ such that if
$$
    \|f(x,u(x),Du(x))\|_{L^p(B_1)} < \eta \quad \text{and} \quad \sup_{x \in B_1}\mbox{osc}_F(x,0) < \eta,
$$
for every $\eta \leq \eta_0$, then, there exists $M \in \mathcal{F}^2_{n,\lambda,\Lambda}$ such that
$$
\left\|u - P_M\right\|_{L^\infty\left(B_{r_0}\right)} \leq r_0^{2+\epsilon_0}.
$$
\end{lemma}
\begin{proof}
Assume, seeking a contradiction, that we can find $r_*>0$ and sequences $(u_k,f_k) \subset C (\overline{B}_1 ) \times L^p(B_1)$ and operators $F_k$ such that
\begin{enumerate}[label=(\roman*)]
    \item $u_k$ is normalized;

    \item $ 0 \in \mathcal{C}(u_k)$;

    \item $\displaystyle \| f_k(x,u_k,Du_k) \|_{L^p(B_1)} < 1/k$ and $ 0 < \sup_{x \in B_1} \mbox{osc}_{F_k}(x,0) < 1/k$;

    \item $F_k(x,D^2u_k) = f_k(x,u_k,Du_k)$ in $B_1$,
\end{enumerate}
however, 
\begin{equation}\label{distance contr assumption}
    \mbox{dist}_{B_{r_*}}\left[u_k, \mathcal{F}^2_{n,\lambda,\Lambda}\right] \geq r_*^{2+\epsilon_0}
\end{equation} 
for all $k\ge 1$.

From \cite{Teix1}, $\{u_k\}_{k \in \mathbb{N} } \in C_{loc}^{1,1-\frac{n}{p}} (B_1)$. Thus, passing to a subsequence if necessary, 
$$
    (u_k,Du_k) \to (u_\infty, Du_\infty)
$$
locally uniform in $L^\infty(B_1) \times L^\infty(B_1)$, which readily implies that $0 \in \mathcal{C}(u_\infty)$. Moreover, by uniform ellipticity and smallness assumption on the coefficients, $F_k \to F_\infty$ through a further subsequence and $f_k \rightarrow 0$. By stability results,
$$
   F_\infty(D^2 u_\infty) = 0 \quad \mbox{in} \quad B_{3/4}.
$$
Since $u_\infty$ is a solution to a uniformly elliptic equation with constant coefficients, the \textit{a priori} estimates assumption implies that $u_\infty$ is in $C^{2,\beta_*}$ satisfying
$$
    \|u_\infty\|_{C^{2,\beta_*}(B_{1/2})} \leq C^*.
$$
In particular, since $0 \in \mathcal{C}(u_\infty)$, it holds that
$$
    \left|u_\infty(x) - \frac{1}{2}D^2u_\infty(0) x\cdot x\right| \leq C^*|x|^{2+\beta_*} \quad \mbox{in} \quad B_{1/4}.
$$
Therefore, for large $k$ and $x \in B_{r_0}$, we have
\begin{eqnarray*}
    \left|u_k(x) - \frac{1}{2}D^2u_\infty(0)x \cdot x \right| & \leq & |u_k(x) - u_\infty(x)| + \left|u_\infty(x) - \frac{1}{2} D^2u_\infty(0) x\cdot x\right| \\
    & \leq & \frac{1}{2} r_0^{2+\epsilon_0} + C^*|x|^{2+\beta_*}\\
    & \leq & \frac{1}{2} r_0^{2+\epsilon_0} + C^*r_0^{2+\beta_*}\\
    &  <   & r_0^{2+\epsilon_0},
\end{eqnarray*}
for $r_0$ small enough such that
$$
    C^* r_0^{\beta_* - \epsilon_0} < \frac{1}{2}.
$$
However, since $P_{D^2 u_\infty(0)} \in \mathcal{F}^2_{n,\lambda,\Lambda}$, we get a contradiction to \eqref{distance contr assumption}.
\end{proof}

The proof of Theorem \ref{l_infinity} relies on an iterative scheme of the approximation lemma above to reach the aimed $C^{2,\epsilon_1}$-estimate of solutions to \eqref{maineq} at the origin as long as
\begin{equation}\nonumber
    p > \frac{n(m+\gamma+1)}{2m+\gamma}.
\end{equation}
It is interesting to note that if $m=\gamma=0$, then this becomes $p>n$ and no improvement can be assured. On the other hand, if $m>1$ and $\gamma=0$, then $p>n$ is enough so that the inequality is true, and an improvement is assured.

\begin{proposition}\label{iteration c2 alfa small exponent}
Let $u \in C (\overline{B}_1  )$ be as in Theorem \ref{l_infinity}. Then, there exists a universal constant $C$, such that
\begin{equation}\label{decay estimate for c2alfa}
    \sup\limits_{B_t(x_0)} \left|u(x) - \frac{1}{2}M(x-x_0) \cdot (x-x_0)\right| \le C t^{2 + \epsilon_0},
\end{equation}
for any $x_0 \in \mathcal{C}(u)$, $t<1/8$ and
\begin{equation}\label{escolha 1 p/ eps_0}
    \epsilon_0 \coloneqq \min \left\{\frac{(m+\gamma)\left(1-\frac{n}{p}\right) + m-\frac{n}{p}}{2}, \, \tau^-, \, \beta_*^-  \right\}.
\end{equation}
\end{proposition}
\begin{proof}
First, by a translation argument, we may assume $x_0 = 0 \in \mathcal{C}(u)$. We note that by Assumption \eqref{inte C^{2,alp} estimates}, equation $F(0,D^2 h + N) = c$ also has $C^{2,\beta_*}$ interior estimates with constant $\Tilde{\Theta}$, depending on $\Theta$ and $|c|$, for any matrix $N$ satisfying $F(0, N) = c$.

Our strategy now is to show, for a radius $r> 0$ to be chosen,  the existence of a sequence of matrices $M_k \in \mbox{Sym}(n)$ such that 
\begin{equation}\label{iteration of lemma 2}
     \left \{
        \begin{array}{rll}
            \sup\limits_{B_{r^{k+1}}} \left| u(x) -  \frac{1}{2}M_k x\cdot x \right| &\leq& r^{k(2 + \epsilon_0)}, \\ 
            \| M_{k}-M_{k-1} \| &\le& C r^{k \epsilon_0}, \quad  \text{and} \\ 
            F(0, M_k) &=& 0.
    \end{array}
    \right.
\end{equation}
for a universal constant $C > 0$.

For $k = 0$, we proceed by choosing $M_0 = M_{-1} = 0$, and \eqref{iteration of lemma 2} is true by the fact that $u$ is normalized and $F(0,0) = 0$. Now we assume that \eqref{iteration of lemma 2} is verified for $1, \dots, k$. Let $v_k \colon B_{1} \to \mathbb{R}$ be defined as
$$
    v_k(x) : = \frac{u(r^{k+1} x) - P_{M_k}\left(r^{k+1}x \right)}{r^{k(2 + \epsilon_0)}} .
$$
As before, one can easily check that $v_k$ is a normalized solution to 
$$
    F_k(z,D^2 v_k) = f_k(z,v_k,Dv_k),
$$
where
$$
    F_k(z,N) = r^{2-k\epsilon_0}F(r^{k+1} z, r^{-(2 - k \epsilon_0)}N + M_k) - r^{2-k\epsilon_0}F(r^{k+1}z, M_k), 
$$
and $f_k(z,s,\xi)$ is defined to be equal to
\begin{align*}
    r^{2-k\epsilon_0}f\left(r^{k+1}z, r^{k(2+\epsilon_0)}s + 
    P_{M_k}\left(r^{k+1}x \right)
    ,r^{k(1+\epsilon_0)-1}\xi + DP_{M_k}\left(r^{k+1}x \right)\right)\\
                  -  r^{2-k\epsilon_0}F\left(r^{k+1} z,M_k\right)
\end{align*}
Observe that $F_k$ is $(\lambda,\Lambda)$-elliptic, $F_k(z, 0) = 0$ and $F_k(0, D^2 w) = 0$ has $C^{2,\beta_*}$ interior estimates, since this equation is equivalent to 
$$
    F(0, D^2  (r^{-2 + k\epsilon_0} w ) + M_k) = 0, 
$$
and $F(0, M_k) = 0$. Moreover, by \textit{a priori} $C^{2,\beta_*}$ estimates, it also holds 
\begin{eqnarray*}
    \mbox{osc}_{F_k}(x,0) & = &\displaystyle  \sup\limits_{M \in \mbox{Sym}(n)} \left| \frac{F_k(x, M) - F_k(0,M)}{1 + \|M\|} \right|\\
    & \leq &  C_0 r^{2 -k\epsilon_0}\mbox{osc}_F\left(r^{k+1} x,0\right),
\end{eqnarray*}
and so, by Assumption \ref{holder continuity of the coeffs}, it holds
\begin{equation}\label{iterat osc smallness}
    \mbox{osc}_{F_k}(x,0) \leq C_1r^{k(\tau - \epsilon_0)}. 
\end{equation}    
By optimal regularity estimates for $u$ and the fact that $0 \in \mathcal{C}(u)$, there holds
$$
    \sup\limits_{B_{t}} \left\{|u|, t|Du| \right\} \le C't^{2-\frac{n}{p}},
$$
where $C'>0$ is an universal constant. To simplify notation, let us define
$$
    A_k(x) \coloneqq r^{2-k\epsilon_0}F\left(r^{k+1}x,M_k\right).
$$
Therefore, we estimate
\begin{eqnarray*}
    |f_k(x,v_k,Dv_k)| & = & r^{2-k\epsilon_0}\left|f\left(r^{k+1}x,u\left(r^{k+1}x\right),Du\left(r^{k+1}x\right)\right) + A_k(x)\right|\\
    & \leq & r^{2-k\epsilon_0}\, q\left(r^{k+1}x\right)\,\left|u\left(r^{k+1}x\right)\right|^m   \left|Du\left(r^{k+1} x\right)\right|^{\gamma} \\
    & &+ |A_k(x)|  \\
    & \leq & C_2 \left( q\left(r^{k+1}x\right)r^{k\left(m\left(2-\frac{n}{p}\right) + \gamma\left(1-\frac{n}{p}\right)-\epsilon_0\right)} + r^{k(\nu - \epsilon_0)} \right), 
\end{eqnarray*}
which implies
\begin{equation}\label{inter RHS smallnes}
    \left\|f_k\right\|_{L^p(B_1)} \leq C_3\left(\|q\|_{L^p(B_1)}r^{k\left(m\left(2-\frac{n}{p}\right) + \gamma\left(1-\frac{n}{p}\right)-\epsilon_0 - \frac{n}{p}\right)} +  r^{k(\nu - \epsilon_0)} \right)
\end{equation}
where we have further used Assumption \ref{holder continuity of the coeffs}, \eqref{RHS_growth}, $F(0, M_k)=0$, and we are abusing notation where $f_k = f_k(x,v_k,Dv_k)$. Next, we choose $\epsilon_0$ as in \eqref{escolha 1 p/ eps_0}. Recall that $\epsilon_0 > 0$ due to \eqref{restriction for p}. This choice is so that
\begin{eqnarray*}
    \max \left\{\sup_{x \in B_1} \mbox{osc}_{F_k}(x,0)\,, \left\|f_k(x,v_k(x),Dv_k(x))\right\|_{L^p(B_1)} \right\}\\
    \leq C_4\left(\|q\|_{L^p(B_1)}r^{\frac{k  \left((m+\gamma)\left(1-\frac{n}{p}\right) + m-\frac{n}{p}\right)}{2}} +  r^{k(\tau - \tau^-)} \right),
\end{eqnarray*}
and so we can choose $r$ small enough so that \eqref{iterat osc smallness} and \eqref{inter RHS smallnes} satisfies a smallness regime. As a consequence, $v_k$ is entitled to Lemma \ref{Approx lemma} and we can find $\tilde{M}_k \in \mathcal{F}^2_{n,\lambda,\Lambda}$ such that 
$$
    \left\| v_k - P_{\tilde{M}_k} \right\|_{L^{\infty}\left(B_r\right)} \leq r^{2+\epsilon_0} .
$$
Scaling back to $u$, we have 
\begin{eqnarray}
    r^{2 + \epsilon_0} &\ge & \left\| v_k - P_{\tilde{M}_k} \right\|_{L^{\infty}\left(B_r\right)} \nonumber \\
    &=& \sup\limits_{B_{r}} \left| \frac{u(r^{k+1} x)  - P_{M_k}\left(r^{k+1}x \right) - r^{-2 +k\epsilon_0}P_{\tilde{M}_k}\left(r^{k+1}x \right)}{r^{k(2 + \epsilon_0)}} \right| \nonumber \\
    &=& \sup\limits_{B_{r}} \left| \frac{u(r^{k+1} x)  - P_{M_{k+1}}\left(r^{k+1}x \right)}{r^{k(2 + \epsilon_0)}} \right| \nonumber,
\end{eqnarray}
where $M_{k+1} = M_k + r^{-2}r^{k\epsilon_0} \tilde{M}_k$. Finally, this implies
$$
    \sup\limits_{B_{ r^{k+2}}} \left| u(x) - P_{M_{k+1}}\left(x \right) \right| \le r^{(k+1)(2 + \epsilon_0)} 
$$
and 
$$
    \| M_{k+1}-M_{k} \| = \| r^{-2}r^{k\epsilon_0} \tilde{M}_k \|  \le r^{-2}C r^{k \epsilon_0} = C r^{k \epsilon_0}, 
$$
since $\tilde{M}_k$ is universally bounded by the $C^{2,\beta_*}$ \textit{a priori} estimates. This concludes the induction step. By \eqref{iteration of lemma 2}, we obtain that
$$
    M_k \to M \quad \mbox{in Sym(n)},
$$
for some symmetric matrix $M$. Moreover,
$$
    |M_k - M| \leq \frac{C}{1-r^{\epsilon_0}}r^{k\epsilon_0}.
$$
Next, given $t <r^2$, there exists $k \in \mathbb{N}$ such that $r^{k+2} \leq t \leq r^{k+1}$. Therefore, if $x \in B_t$, then
\begin{eqnarray*}
       \displaystyle \left| u(x) - P_{M}\left(x \right) \right|   & \leq  & \displaystyle \left| u(x) - P_{M_k}\left(x \right) \right| +  \left|P_{M_k - M}\left(x \right) \right| \\[0.4cm]
       & \leq & \displaystyle r^{k(2+\epsilon_0)} + \frac{C}{1-r^{\epsilon_0}}r^{k\epsilon_0}t^2 \\[0.4cm]
       & \leq & \left(  r^{-2(2+\epsilon_0)} + \frac{C}{1-r^{\epsilon_0}}r^{-2\epsilon_0} \right)\, t^{2+\epsilon_0}.
\end{eqnarray*}
\end{proof}

\begin{proof}[Proof of Theorem \ref{l_infinity}]
The proof follows the same lines as in the proof of Proposition \ref{iteration c2 alfa small exponent}, except that now, estimate \eqref{decay estimate for c2alfa} implies, by Lemma \ref{growth est sec-prd version},
\begin{equation}\label{new decay for gradient at critical points}
    \sup_{x \in B_\rho} \left\{|u(x)|, \rho |Du(x)| \right\} \leq \overline{C} \rho^2,
\end{equation}
for $\rho \in (0,1/2)$. And so, by carefully following the lines of the proof, one notices that we can improve the choice of \eqref{escolha 1 p/ eps_0} to the new exponent
\begin{equation}\label{new C2alfa exp}
    \epsilon_1 \coloneqq \min \left\{2m+\gamma-\frac{n}{p}, \tau, \beta_*  \right\}^-.
\end{equation}
The rest of the proof follows seamlessly.
\end{proof}

\subsection{Gain of regularity at inflection points}

Theorem \ref{l_infinity} also can be understood as a $C^{1,\alpha}$ implies $C^{2,\alpha}$ at points in the set $\mathcal{C}(u)$. In contrast to Section \ref{sct crit pt reg case p}, this result cannot be iterated indefinitely. The reason for such is because it is not always true that
\begin{equation}\label{crit pts are inflection pts}
     \mathcal{C}(u) \subset \{D^2u = 0\}.
\end{equation}
Observe that this can be understood in a pointwise sense due to Theorem \ref{l_infinity}, which is nontrivial information since solutions of \eqref{maineq} are, at best, $C^{1,1-\frac{n}{p}}$.

It is worthwhile to mention that the asymptotic analysis can be actually carried away as long \eqref{crit pts are inflection pts} is true, as follows
\begin{proposition}\label{recursive exponent c2 alfa}
Let $u \in C (\overline{B}_1  )$ be as in Theorem \ref{l_infinity} and assume \eqref{crit pts are inflection pts} is in force. If
$$
    \sup_{B_t(x_0)}|Du| \leq C_kt^{1+\epsilon_k},
$$
then
$$
    \sup_{B_t(x_0)}|Du| \leq C_{k+1}t^{1+\epsilon_{k+1}},
$$
where
$$
    \epsilon_{k+1} \coloneqq \min \left\{\frac{(m+\gamma)(1+\epsilon_k) + m-\frac{n}{p}}{1+\theta},\beta_*^- \right\},
$$
for any $x_0 \in \mathcal{C}(u)$ and a fixed parameter $\theta>0$.
\end{proposition}
\begin{proof}
The proof is followed by an induction argument, which we only sketch. First, in order for the argument to hold, we need an approximation lemma, as \ref{Approx lemma} that is stable under hessian degenerate points. By taking into account the Dini continuity of the coefficients of the diffusion operator, the results from \cite{JK}, assure us of such stability.

As did before, we assume $x_0 = 0 \in \mathcal{C}(u)$. As in the proof of Theorem \ref{l_infinity}, we define
$$
    v_j(x) \coloneqq \frac{u(r^{j+1}x)}{r^{j(2+\epsilon_{k+1})}},
$$
which solves
$$
    F_j(z,D^2 v_j) = f_j\left(z,v_j,Dv_j\right),
$$
where
$$
    F_j(z,N) = r^{2-j\epsilon_{k+1}}F(r^{j+1} z, r^{-(2 - j \epsilon_{k+1})}N), 
$$
and
\begin{eqnarray*}
    f_j(z,s,\xi) & = & r^{2-j\epsilon_{k+1}} f\left(r^{j+1}z, r^{j(2+\epsilon_{k+1})}s ,r^{j(1+\epsilon_{k+1})-1}\xi \right).
\end{eqnarray*}
Observe that by Assumption \ref{convexity}, it holds that
\begin{equation}\nonumber
    \left\|f_j(x,v_j(x),Dv_j(x))\right\|_{L^p(B_1)} \leq C\left(\|q\|_{L^p(B_1)}r^{j(m(2+\epsilon_k) + \gamma(1+\epsilon_k)-\epsilon_{k+1} - \frac{n}{p})} \right).
\end{equation}
The choice of $\epsilon_{k+1}$ is so that we can assure the smallness regime required to apply the approximation lemma (Lemma \ref{Approx lemma}) and the rest of the proof follows similarly as in the proof of Theorem \ref{l_infinity}.
\end{proof}
\begin{remark}
It is worth noting that, while in Theorem \ref{l_infinity}, the level of regularity in the coefficients imposes constraints on the regularity of solutions, in the case investigated in Proposition \ref{recursive exponent c2 alfa}, where  
$$
    \mathcal{C}(u) \subset \{D^2 u = 0 \},
$$
we manage to bypass such a barrier. 
\end{remark}
Finally, we perform an asymptotic analysis of the recursive exponent $\epsilon_k$.
\begin{proposition}\label{asymptotics c2 alfa case}
Let $\{\epsilon_k\}_{k \in \mathbb{N}}$ be the nondecreasing recursive sequence defined as in Proposition \ref{recursive exponent c2 alfa}. Then,
$$
    \epsilon_{m,\gamma,n,p,\theta} \coloneqq \lim_{k \to \infty} \epsilon_k
$$
exists, and
$$
    \epsilon_{m,\gamma,n,p,\theta} = \min \left\{\frac{2m + \gamma - \frac{n}{p}}{(1+\theta - (m+\gamma))_+} ,\beta_*^- \right\}
$$
\end{proposition}
\begin{proof}
As we did in the proof of Proposition \ref{asymptotic analysis case 1,alfa}, we define
$$
    k_0 \coloneqq \sup\left\{i \in \mathbb{N}\, \Big{\vert}\, \epsilon_{i+1} = \frac{(m+\gamma)(1+\epsilon_{i}) + \left(m - \frac{n}{p}\right)}{1+\theta}\right\}.
$$
By definition of $k_0$, it holds 
$$
    \epsilon_{k+1} = \frac{(m+\gamma)(1+\epsilon_{k}) + \left(m -  \frac{n}{p}\right)}{1+\theta}
$$
for every $k \leq k_0$. Letting $\eta \coloneqq 2m + \gamma - \frac{n}{p}$, we can rewrite it as follows
\begin{eqnarray*}
    (1+\theta) \epsilon_{k+1} & = & \eta + \left(\frac{m+\gamma}{1+\theta} \right) ((1+\theta)\epsilon_k)\\
                        & = & \eta + \left(\frac{m+\gamma}{1+\theta} \right) \left(\eta + \left(\frac{m+\gamma}{1+\theta} \right) ((1+\theta)\epsilon_{k-1}) \right)\\
                        & = & \eta + \eta \left(\frac{m+\gamma}{1+\theta} \right) + \left(\frac{m+\gamma}{1+\theta} \right)^2 ((1+\theta)\epsilon_{k-1})\\
                        & = & \eta \left(\sum_{j=0}^{i} \left(\frac{m+\gamma}{1+\theta} \right)^j \right) + \left(\frac{m+\gamma}{1+\theta} \right)^{i+1}((1+\theta) \epsilon_{k-i}),
\end{eqnarray*}
for $i \in \{2,\cdots, k \}$. Therefore, it follows that
$$
    \epsilon_{k+1} = \frac{\eta}{1+\theta} \left(\sum_{j=0}^{k} \left(\frac{m+\gamma}{1+\theta} \right)^j \right) + \left(\frac{m+\gamma}{1+\theta} \right)^{k+1} \epsilon_{0}.
$$
Now, if $k_0 = \infty$, then we claim that $m+\gamma < 1+\theta$. Indeed, if $m+\gamma \geq 1+\theta$, then
$\epsilon_k \to \infty$, which is a contradiction, since $\epsilon_k \leq \beta_*^-$ for every $k \in \mathbb{N}$. Now, if
$$
    \frac{m+\gamma}{1+\theta} < 1,
$$
then, $\epsilon_k$ converges to a geometric series whose sum is given by
$$
    \lim_{k \to \infty} \epsilon_k  = \frac{2m + \gamma - \frac{n}{p}}{1 + \theta - (m+\gamma)}.
$$
If $k_0 < \infty$, then, by definition of $k_0$, it holds $\epsilon_{k_0 + 2} = \beta_*^-$, and the proposition is proved.
\end{proof}

Finally, we gather all results in order to give the proof of Theorem \ref{Lp p large}.
\begin{proof}[Proof of Theorem \ref{Lp p large}]
We assume $x_0 = 0 \in \mathcal{C}(u) \cap \mathcal{C}_2(u)$. Given $\theta \in (0,1)$, we apply Proposition \ref{recursive exponent c2 alfa} in order to obtain
$$
    \sup_{B_t(x_0)}|Du| \leq C_{k,\theta}t^{1+\epsilon_{k,\theta}}.
$$
By Proposition \ref{asymptotics c2 alfa case},
$$
  \lim_{k \to \infty} \lim_{\theta \to 0} \epsilon_{k,\theta} = \min \left\{ \frac{2m+\gamma - \frac{n}{p}}{1 - (m+\gamma)}, \beta_*^- \right\},  
$$
and so, for $k$ large and $\theta$ small, depending on $m$, $\gamma$, $n$ and $p$, it follows that
$$
    \epsilon_{k,\theta} \geq \min \left\{ \frac{2m+\gamma - \frac{n}{p}}{1 - (m+\gamma)}, \beta_* \right\}^-,
$$
from which follows the desired.
\end{proof}

\section{Regularity below the dimension threshold}\label{Sct Reg below n}

In this section, we give the proof of Theorems \ref{reg below dimension} and \ref{reg below dimension interplay}. The starting point of the proof is the sharp $C^{0,\frac{n-2\nu}{n-\nu}}$ regularity estimates obtained in \cite{Teix1}. The key novelty in this section is that we modify Assumption \ref{convexity} as to only require that  RHS has a priori bounds in the $L^{n-\nu}$ space, where $\nu \in (0,\varepsilon_E)$ and $\varepsilon_E$ stands for the Escauriaza exponent, see \cite{ESC}. Without further assumptions on how close $n-\nu$ is from the dimension, $n$, merely improved H\"older estimates are available. However, when there is an interplay between the decay in the zeroth term and the amount of integrability on the RHS, we surpass the previous H\"older regularity regime. We remark that the forthcoming proofs follow the same strategy as before with minor amendments. We bring it here for the reader's convenience.

In the following, $\nu$ will always denote an exponent in the range $(0,\varepsilon_E)$, where $\varepsilon_E$ stands for the Escauriaza exponent.
\subsection{Improved H\"older estimates}
We begin with a simple flatness lemma, which states our problem, up to scaling, is uniformly close to functions in $\mathcal{F}_{n,\lambda,\Lambda}$. Recall from Assumption \ref{convexity}, the RHS satisfies
$$
    |f(x,s,\xi)| \leq q(x)|s|^m\min\{1,|\xi|^\gamma \} \leq q(x)|s|^m.
$$
Since we are dealing with normalized functions, we may assume that
\begin{equation}\label{decay RHS below dim}
    |f(x,s,\xi)| \leq q(x).
\end{equation}
We will be using such a feature in the following result.
\begin{lemma}\label{flat lemma holder}
Let $u \in C(\overline{B}_1)$ be a normalized viscosity solution of
$$
    \tilde{F}_\mu(x,D^2u) = f(x,u,Du) \quad \mbox{in} \quad B_1.
$$
Assume $0 \in \mathcal{C}_0(u)$. Given $\delta>0$, there exists $\eta = \eta(\delta,n,\lambda,\Lambda)$ such that if 
$$
    \|q\|_{L^{n-\nu}(B_1)} \leq \eta \quad \mbox{and} \quad \mu < \eta,
$$
where $q$ is from \eqref{decay RHS below dim}, then there exists $h \in \mathcal{F}_{n,\lambda,\Lambda}$ such that $0 \in \mathcal{C}_0(h)$ and
$$
    \|u-h\|_{L^\infty(B_{1/2})} < \delta.
$$
\end{lemma}
\begin{proof}
Assume, seeking a contradiction, that for some $\delta_0 > 0$, there exists a sequence $(u_k, q_k, \mu_k)_{k \in \mathbb{N}} \subset C (\overline{B}_1 ) \times L^p(B_1) \times \mathbb{R}^+$ satisfying
\begin{enumerate}[label=(\roman*)]

    \item $u_k$ is normalized;

    \item $0 \in \mathcal{C}_0(u_k)$;

    \item $\displaystyle \max \left\{ \|q_k\|_{L^{n-\nu}(B_1)}, \mu_k \right\} \leq \frac{1}{k} $;

    \item $\tilde{F}_{\mu_k} (x, D^2 u_k ) = f_k(x, u_k, D u_k)$ in $B_1$;
\end{enumerate}
however, 
\begin{equation}\label{Appr Lem 
contrad eq case L^p}
    \mbox{dist}\left[u_k, \mathcal{F}_{n,\lambda, \Lambda}\right] \geq \delta_0, 
\end{equation} 
for all $k\ge 1$. By our assumptions on $f$ and the diffusion operator, we have $\{ u_k \}_{k \in \mathbb{N} } \in C_{loc}^{0,\frac{n-2\nu}{n-\nu}} (B_1)$, with universal estimates. Therefore, passing to a subsequence if necessary, we obtain $u_k \to u_\infty$
locally uniform in $L^\infty(B_1)$; in particular we deduce that $0 \in \mathcal{C}_0(u_\infty)$. Moreover, through a further subsequence in necessary, we obtain $\tilde{F}_{\mu_k} \to F'$ locally uniformly on $B_1 \times \mbox{Sym}(n)$ and  $f_k \to 0$. Thus, by stability results in the theory of viscosity solutions, we have
$$
    F'(D^2 u_{\infty}) = 0 \quad \text{in} \quad B_{3/4},
$$
for some $(\lambda,\Lambda)-$elliptic operator $F'$, which contradicts \eqref{Appr Lem contrad eq case L^p} for $k$ sufficiently large. 
\end{proof}
As mentioned before, the starting point is that if we assume $0 \in \mathcal{C}_0(u)$ and $u$ is a normalized viscosity solution of \eqref{maineq}, then by the regularity estimates from \cite{Teix1}, we have
$$
    \sup_{B_t}|u| \leq Ct^{\frac{n-2\nu}{n-\nu}}.
$$
For a positive $\theta$, we define
\begin{equation}\label{first exponent below dim}
    \varepsilon_1 \coloneqq \min \left\{\frac{\frac{n-2\nu}{n-\nu} + m\frac{n-2\nu}{n-\nu}}{1+\theta},1^- \right\}.
\end{equation}
\begin{proposition}
Let $u \in C (\overline{B}_1)$ be as in Theorem \ref{reg below dimension}. If
$$
    \sup_{B_t(x_0)} |u| \leq C_0t^{\frac{n-2\nu}{n-\nu}},
$$
then
$$
    \sup_{B_t(x_0)} |u| \leq C_1t^{\epsilon_1},
$$
for any $x_0 \in \mathcal{C}_0(u)$, where $\epsilon_1$ is as defined in \eqref{first exponent below dim}.
\end{proposition}
\begin{proof}
Assume $x_0 = 0$. We will prove that there is a radius $r$, to be chosen in the sequel, such that
\begin{equation}\label{decay induction argument below dim}
    \sup\limits_{B_{r^{k}}} \left|u(x)\right| \leq r^{k\epsilon_1}
\end{equation}
for a universal constant $C > 0$ for every $k \in \mathbb{N}$. We proceed by induction. The case $k=0$ follows since $u$ is normalized. We assume \eqref{decay induction argument below dim} is verified for $1,\cdots,k$. Let $v_k\colon B_1 \to \mathbb{R}$ be defined as
$$
    v_k(x) \coloneqq \frac{u(r^kx)}{r^{k\epsilon_1}}.
$$
This function solves
$$
    F_k(z,D^2v_k) = f_k(z,Dv_k),
$$
where
\begin{eqnarray*}
    F_k(z,N) &=& r^{k(2-\epsilon_0)}F(r^{k} z, r^{-k(2 -  \epsilon_0)}N)\\
    f_k(z,\xi)& = &r^{k(2-\epsilon_0)}f\left(r^{k}z, u(r^kx)
    ,r^{k\epsilon_1}\xi \right).
\end{eqnarray*}
Due to Assumption \ref{convexity}, and since $0 \in \mathcal{C}_0(u)$, we have
$$
    |f_{k}(z,\xi)| \leq r^{k(2-\epsilon_1)}q(r^kx)\left|u(r^kx)\right|^m \leq r^{k\left(2-\epsilon_1 + m\,\frac{n-2\nu}{n-\nu}\right)}q(r^k x) = q_k(x).
$$
Observe that
\begin{equation}
    \|q_k\|_{L^{n-\nu}(B_1)} \leq r^{k\left(2-\epsilon_1 + m\,\frac{n-2\nu}{n-\nu} - \frac{n}{n-\nu}\right)}\|q\|_p,
\end{equation}
and so we can pick $r$ small enough so that it lies in the smallness regime of Lemma \ref{flat lemma holder}. Since $F_k$ is $(\lambda,\Lambda)$-elliptic, $v_k$ is normalized and the RHS is small in $L^{n-\nu}$ norm, we can apply Lemma \ref{flat lemma holder} to obtain $h \in \mathcal{F}_{n,\lambda,\Lambda}$ such that
$$
    \|v_k-h\|_{L^\infty(B_{1/2})} < \delta.
$$
Therefore, as $h$ enjoy $C^{1,\alpha_*}$ estimates and $0 \in \mathcal{C}_0(h)$, we have
\begin{eqnarray*}
    |v_k| &\leq& |v_k-h| + |h|\\
                       & \leq & \delta + Cr\\
                       & \leq & r^{\epsilon_1},
\end{eqnarray*}
for $\delta =  r^{\epsilon_1}/2$. Scaling back to $u$ we get
$$
    \sup_{B_{r^{k+1}}}|u| \leq r^{(k+1)\epsilon_1},
$$
and the Proposition is proved once we realize that for a given $t \in (0,1/2)$, there is $k \in \mathbb{N}$ such that $r^{k+1}\leq t \leq r^k$, and so
$$
    \sup_{B_t}|u| \leq C_1t^{\epsilon_1},
$$
for a universal constant $C_1$.
\end{proof}
This argument can be repeated indefinitely, yielding the following result:
\begin{proposition}
Let $u \in C (\overline{B}_1)$ be as in Theorem \ref{reg below dimension}. If
$$
    \sup_{B_t(x_0)} |u| \leq C_0t^{\epsilon_k},
$$
then
$$
    \sup_{B_t(x_0)} |u| \leq C_1t^{\epsilon_{k+1}},
$$
for any $x_0 \in \mathcal{C}_0(u)$, where $\epsilon_{k+1}$ is defined as
$$
    \epsilon_{k+1} \coloneqq \min \left\{ \frac{\frac{n-2\nu}{n-\nu} + m\epsilon_k}{1+\theta},1^- \right\}.
$$
\end{proposition}
Finally, we give the 
\begin{proof}[Proof of Theorem \ref{reg below dimension}]
The proof is followed by an asymptotic analysis of the exponents $\epsilon_k$. It is enough to assume $m\leq 1$, otherwise if $m>1$, then, $\epsilon_k > \frac{\nu}{n-\nu}$, for large $k$, and so
$$
    \frac{n-2\nu}{n-\nu} + m\epsilon_k > \frac{n-2\nu}{n-\nu} + \frac{\nu}{n-\nu} = 1.
$$
Therefore, for $\theta$ small enough, we have
$$
    \frac{\frac{n-2\nu}{n-\nu} + m\epsilon_k}{1+\theta} > 1,
$$
and so we are done. Assume $m \leq 1$. We can also assume
$$
    \epsilon_{k+1} = \frac{\frac{n-2\nu}{n-\nu} + m\epsilon_k}{1+\theta}
$$
for every $k \in \mathbb{N}$, otherwise we are done. Passing to the limit as $k \to \infty$, we have
$$
    (1+\theta - m)\epsilon_\infty = \frac{n-2\nu}{n-\nu},
$$
and so, by making $\theta \to 0$, the Theorem is proved.
\end{proof}
\subsection{Gradient continuity at vanishing points}
We investigate $C^{1+}$ regularity estimates at vanishing points. The first nontrivial step is the following
\begin{proposition}\label{improvement below dimension}
Let $u \in C (\overline{B}_1)$ be as in Theorem \ref{reg below dimension interplay}. If
$$
    \frac{m\, n}{2m + 1} > \nu,
$$
then $u$ is differentiable at $x_0$ and there holds
$$
    \sup_{B_t(x_0)}|u(x) - Du(x_0) \cdot (x-x_0)| \leq C\,t^{1+\epsilon_0}
$$
for any $x_0 \in \mathcal{C}_0(u)$, where
$$
    \epsilon_0 \coloneqq \min\left\{\frac{m\,n - \nu(2m+1)}{n-\nu}, \alpha_* \right\}^-.
$$
\end{proposition}
\begin{proof}
Assume $x_0 = 0$. Our strategy now is to show, for a radius $r> 0$ to be chosen,  the existence of a sequence of vector $\xi_k \in \mathbb{R}^n$ such that 
\begin{equation}\label{iteration of lemma 0}
     \left \{
        \begin{array}{rll}
            \sup\limits_{B_{r^{k}}} \left| u(x) - \xi_k \cdot x \right| &\leq& r^{k(1 + \epsilon_0)}, \\ 
            \| \xi_{k}-\xi_{k-1} \| &\le& C r^{k \epsilon_0},
    \end{array}
    \right.
\end{equation}
for a universal constant $C > 0$.

For $k = 0$, we proceed by choosing $\xi_0 = \xi_{-1} = 0$, and \eqref{iteration of lemma 0} is true by the fact that $u$ is normalized. Now we assume that \eqref{iteration of lemma 0} is verified for $1, \dots, k$. Let $v_k \colon B_{1} \to \mathbb{R}$ be defined as
$$
    v_k(x) : = \frac{u(r^{k} x) - \xi_k \cdot \left(r^{k}x \right)}{r^{k(1 + \epsilon_0)}} .
$$
One can easily check that $v_k$ is a normalized solution to 
$$
    F_k(z,D^2 v_k) = f_k(z,Dv_k),
$$
where
\begin{eqnarray*}
    F_k(z,N) &=& r^{k(1-\epsilon_0)}F(r^{k} z, r^{-k(1 -  \epsilon_0)}N)\\
    f_k(z,\xi)& = &r^{k(1-\epsilon_0)}f\left(r^{k}z, u(r^kx)
    ,r^{k\epsilon_0}\xi + \xi_k \right).
\end{eqnarray*}
Due to Assumption \ref{convexity}, and since $0 \in \mathcal{C}_0(u)$, we have
$$
    |f_{k}(z,s,\xi)| \leq r^{k(1-\epsilon_0)}q(r^kx)\left|u(r^kx)\right|^m \leq r^{k\left(1-\epsilon_0 + m\,\frac{n-2\nu}{n-\nu}\right)}q(r^k x) = q_k(x).
$$
Observe that
\begin{equation}\label{decay RHS below dim 2}
    \|q_k\|_{L^{n-\nu}(B_1)} \leq r^{k\left(1-\epsilon_0 + m\,\frac{n-2\nu}{n-\nu} - \frac{n}{n-\nu}\right)}\|q\|_p.
\end{equation}
By picking $r$ small enough, we ensure the PDE lies in the smallness regime of the flatness lemma \ref{flat lemma holder}, and thus we obtain $h \in \mathcal{F}_{n,\lambda,\Lambda}$ such that
$$
    \|v_k-h\|_{L^\infty(B_{1/2})} < \delta.
$$
Therefore, as $h$ enjoy $C^{1,\alpha_*}$ estimates and $0 \in \mathcal{C}_0(h)$, we have
\begin{eqnarray*}
    |v_k - Dh(0)\cdot x| &\leq& |v_k-h| + |h - Dh(0)\cdot x|\\
                       & \leq & \delta + Cr^{1+\alpha_*}\\
                       & \leq & r^{1+\epsilon_0},
\end{eqnarray*}
for $\delta =  r^{1+\alpha_*}/2$. Scaling back to $u$ we get
$$
    \sup_{B_{r^{k+1}}}|u - \xi_{k+1}\cdot x | \leq r^{(k+1)(1+\epsilon_0)},
$$
where $\xi_{k+1} = \xi_k + r^{k\epsilon_0}\,Dh(0)$. By universal $C^{1,\alpha_*}$ estimates of $h$, we have $|\xi_{k+1} - \xi_k| \leq Cr^{k\epsilon_0}$ and thus \eqref{iteration of lemma 0} is proven for every $k \in \mathbb{N}$. It is classical that this condition implies $\{\xi_k \}_{k \in \mathbb{N}}$ satisfies the Cauchy condition and is thus convergent. It also follows that $u$ is differentiable at $0$ and $\xi_k \to Du(0)$ and $k \to \infty$. Given $t \in (0,1/2)$, there exists $k \in \mathbb{N}$ such that $r^{k+1}< t \leq r^k$, and so
\begin{eqnarray*}
    \sup_{B_t}|u - Du(0)\cdot x| & \leq & \sup_{B_{r^k}}|u - \xi_k\cdot x| + |\xi_k - Du(0)|r^k\\
                                 & \leq & (C_0+1)r^{k(1+\epsilon_0)} \leq \overline{C}t^{1+\epsilon_0},
\end{eqnarray*}
where we have used
\begin{eqnarray*}
    |\xi_{k+m} - \xi_k| & \leq & \sum_{i=1}^m |\xi_{k+m} - \xi_{k+m-i}|\\
                        & \leq & C\sum_{i=1}^m r^{(k+i)\epsilon_0}\\
                        & \leq & C \frac{r^{\epsilon_0}}{1-r^{\epsilon_0}} r^{k\epsilon_0},
\end{eqnarray*}
and so, passing to the limit as $m \to \infty$, we have 
$$
    |\xi_k - Du(0)| \leq C \frac{r^{\epsilon_0}}{1-r^{\epsilon_0}} r^{k\epsilon_0}.
$$
\end{proof}
As a consequence, if $x_0 \in \mathcal{C}_0(u)$, then
$$
    \sup_{B_t(x_0)}|u(x)| \leq Ct,
$$
for $t<1/2$. Then we can run the algorithm once more to improve the previous regularity exponent, which leads to the proof of Theorem \ref{reg below dimension interplay}.
\begin{proof}[Proof of Theorem \ref{reg below dimension interplay}]
We assume $x_0 = 0$. The proof is similar to the proof of Proposition \ref{improvement below dimension} and we just briefly comment on the main steps. The starting point is that
$$
    \sup_{B_t}|u(x)| \leq Ct,
$$
for $0<t<1/4$. We will construct a sequence as in \eqref{iteration of lemma 0} with $\epsilon_1$ instead of $\epsilon_0$. The main observation is that \eqref{decay RHS below dim 2} becomes
$$
    \|q_k\|_{L^{n-\nu}(B_1)} \leq r^{k\left(1-\epsilon_1 + m - \frac{n}{n-\nu}\right)}\|q\|_p.
$$
The choice of the exponent $\epsilon_1$ is so that we can ensure the smallness regime of such $L^{n-\nu}$ norm.
\end{proof}
We cannot run the asymptotic analysis at the gradient level, as our estimates hold pointwise at points in $\mathcal{C}_0(u)$, and no compactness is assured. In particular, it is not possible to use the gradient decay to make such an improvement.

\section{Regularity at extrema points}\label{STC reg at extrema}
In this section, we prove a regularity result at local extrema points. Hereafter in this section, we assume $p> n-\varepsilon_E$, where, as before, $\varepsilon_E \in (0, \frac{n}{2}]$ is the universal Escauriaza constant. Furthermore, to highlight the robustness of the conclusions outlined in this section it is worth highlighting that herein we don't impose any continuity assumptions on the coefficients. Specifically, the operator $F$ in this section isn't bound by the requirements of Assumption \ref{holder continuity of the coeffs}. Consequently, the available regularity estimates are fundamentally rooted in $C^{0,\delta}$, even for $F$-harmonic functions—namely, viscosity solutions of the equation
$$
    F(x, D^2h) = 0.
$$

At the heart of this section lies the next flatness lemma, a pivotal tool that distinguishes itself from its predecessors. Unlike the prior lemmas, which focused on establishing proximity between the space of solutions and the space of $h$-harmonic functions, this lemma goes a step further, yielding a more robust assertion: solutions closely approximate constant functions. Here is its precise statement.

\begin{lemma}\label{flat lemma extrema}
Let $u \in C(\overline{B}_1)$ be a normalized viscosity solution to
$$
    F(x, D^2u) = f(x),
$$
in $B_1$. Assume $x_0 \in B_{1/2}$ is a local minimum. Then given $t>0$, there exists $s>0$, depending only on dimension, ellipticity,  and $t$, such that if $\|f\|_p < s$, there holds
$$
    \sup\limits_{B_{\frac{1}{10}}(x_0)} \left ( u(x) - u(x_0) \right ) \le t.
$$
\end{lemma}
\begin{proof}
Suppose, seeking a contradiction, the thesis of the Lemma does not hold true. Then, we would find $t_0>0$ and a sequence $(u_k,F_k,f_k,x_k)$ such that
\begin{equation}\label{flat Eq1}
    F_k(x,D^2u_k) = f_k,
\end{equation}
with $u_k$ normalized, $\|f_k\|_{L^p} < 1/k$, and $x_k \in B_{1/2}$ is a local minimum of $u_k$,  but
\begin{equation}\label{flat Eq2}
    \sup\limits_{B_{\frac{1}{10}}(x_k)} \left ( u_k(x)  - u_k(x_k) \right ) \geq t_0.
\end{equation}
By uniform H\"older continuity of $\{u_k\}$, up to a subsequence, we can assume  $x_k \to x_\infty$, $f_k \to 0$, and $u_k \to u_\infty$ uniformly in $B_{2/3}$. It further follows from  uniform convergence that $x_\infty$ is a local minimum of $u_\infty$. Finally, we notice that, in view of \eqref{flat Eq1}, there holds
$$
    \mathcal{M}^+_{\lambda,\Lambda}(D^2u_k) \ge -|f_k| \quad \text{ and } \quad \mathcal{M}^-_{\lambda,\Lambda}(D^2u_k) \le  |f_k|. 
$$
Thus, passing to the limit, we conclude 
$$
    \mathcal{M}^+_{\lambda,\Lambda}(D^2u_\infty) \ge 0 \quad \text{ and } \quad \mathcal{M}^-_{\lambda,\Lambda}(D^2u_\infty) \le  0.
$$
Thus, $u_\infty$ is entitled to the strong maximum principle, which implies $u_\infty \equiv \mbox{const}$. This leads to a contradiction on \eqref{flat Eq2} if we take $k>1$ large enough.
\end{proof}

\begin{proposition} Let $u$ be a normalized viscosity solution of 
$$
    F(x, D^2u) = f(x, u, Du),
$$
and assume $x_0$ is an interior extremum. Then
$$
    \sup\limits_{B_r(x_0)} |u - u(x_0) | \le C r^{M},
$$
where $M = (2-n/p)$.
\end{proposition}
\begin{proof} We will assume, with no loss, that $x_0$ is a local minimum. In the previous Lemma, take 
    $$
        t = \frac{1}{10^M},
    $$
    denote by $s_{1/10^M}$ the corresponding smallness requirement on the $p$-norm of the source term as to assume 
    $$
        \sup\limits_{B_{\frac{1}{10}}(x_0)} \left ( u(x) - u(x_0) \right ) \le \frac{1}{10^M}.
    $$
    After a universal zoom-in, we can assume, with no loss, that $\|f(x, u, Du)\|_p < s_{1/10^M}$. Next we define
    $$
        u_1(x) = 10^M \left(u(x_0 + 10^{-1}x)-u(x_0)\right).
    $$
    This function is normalized, and, because of the appropriate choice of $M$, it is easy to see, as done in the previous sections, that $u_1$ is entitled to the same conclusion of lemma \ref{flat lemma extrema}. Apply recursively, this process will lead to the desired regularity.
\end{proof}
We conclude this section by commenting that if we additionally assume that the minimum point $x_0 \in \mathcal{C}_0(u)$, then the flatness lemma \ref{flat lemma extrema} gives closeness to the zero function. We can repeat the proof of the previous proposition with
$$
    u_{1}(x) = 10^{M_1}u(x_0 + 10^{-1}x) \quad \mbox{with} \quad M_1 = \left(2+Mm - \frac{n}{p}\right)
$$
and use
$$
    \sup_{B_r(x_0)}|u| \leq Cr^M
$$
to improve the decay of the RHS. This argument can be repeated indefinitely, leading to the recursive exponent
$$
    M_{k+1} = \left(2 + M_km - \frac{n}{p} \right).
$$
By previous arguments, we know that if $m\geq1$, then $M_{k+1} \to \infty$, and if $m < 1$, then it leads to the exponent
$$
    M_\infty(1-m) = \left(2- \frac{n}{p}\right).
$$
Therefore, if $m\geq 1$, then $u$ is infinitely times differentiable at a local extrema $x_0 \in \mathcal{C}_0(u)$, with $D^ku(x_0)=0$ for every $k \in \mathbb{N}$. If $m < 1$, then $u$ is $C^{\frac{2-\frac{n}{p}}{1-m}}$ differentiable at $x_0$.
\section{Appendix A. Lipschitz estimates} \label{App A}

We dedicate this appendix to comment on the borderline case where $q \in L^\infty$. As a courtesy to the reader, we bring comprehensive proof of the local Lipschitz regularity via the celebrated Ishii-Lions technique. We bring it in a general setting to apply to as many situations.

We drop the cut-off in Assumption \ref{convexity}, that is we consider viscosity solutions of \eqref{maineq} under the weaker condition
\begin{equation}\label{new assumption RHS}
    |f(z,s,\xi)| \leq q(x)|s|^m |\xi|^\gamma \quad \mbox{for} \quad q \in L^\infty.
\end{equation}
\begin{proposition}\label{lipschitz estimates bounded source}
Let $u \in C (\overline{B}_1)$ be a viscosity solution 
to \eqref{maineq} under Assumptions \ref{unif ellipticity},   \ref{holder continuity of the coeffs} and its RHS satisfies \ref{new assumption RHS}. Then, there exists a constant $C$ depending on $n$, $\lambda$, $\Lambda$, $m$, $\gamma$, $\tau$, $\|q\|_\infty$, $\|u\|_\infty$ and $|F(0,0)|$ such that
$$
    \sup_{x,y \in B_{\frac{1}{2}}}\frac{|u(x) - u(y)|}{|x-y|} \leq C.
$$
\end{proposition}
\begin{proof}
First, we observe that it is enough to consider the case where $m=0$, as we can absorb the zeroth order term with its $L^\infty$ estimate.

Let $\gamma_0$ be a constant such that
$$
    1 > \gamma_0 > \max\left\{\gamma - 1 , 1-\tau \right\},
$$
where $\tau \in (0,1)$ is from Assumption \ref{holder continuity of the coeffs}, and define
\[
\omega(r) = \left\{
     \begin{array}{@{}l@{\thinspace}l}
       \displaystyle  r - \frac{1}{2-\gamma_0}r^{2-\gamma_0} &  \quad r \in [0,1]\\[0.4cm]
       \displaystyle 1 - \frac{1}{2-\gamma_0} & r \geq 1. \\
     \end{array}
   \right.
\]
For constants $\overline{L}, \varrho$ let
$$
    \varphi(x,y) \coloneqq \overline{L} \omega(|x-y|) + \varrho (|x|^2 + |y|^2)
$$
and
$$
    M \coloneqq \sup_{x,y \in \overline{B}_{3/4}} \{u(x) - u(y) - \varphi(x,y) \}.
$$
To prove Lipschitz continuity of $u$, we will show that the quantity $M$ is non-positive. To do so, we assume, by contradiction, that $M>0$. Let $(x_0,y_0)$ be the pair where $M$ is attained. We observe that since $M>0$,
\begin{equation}\label{localization}
    \overline{L}\omega(|x_0 - y_0|) + \varrho (|x_0|^2 + |y_0|^2) < u(x_0) - u(y_0) \leq 2 \|u\|_\infty. 
\end{equation}
Observe that
\begin{equation}\label{loc1}
    \max \{|x_0|,|y_0| \} \leq \sqrt{\frac{2\|u\|_\infty}{\varrho}},
\end{equation}
and so choosing $\varrho$ large enough depending only on $\|u\|_\infty$, we get that both $x_0$ and $y_0$ are interior points. By \cite{CIL92}*{Theorem 3.1}, given $\iota > 0$, we get the existence of matrices $X_\iota$ and $Y_\iota$ such that
\begin{equation}\label{VISC_INEQ_MATRICES}
    \left[
        \begin{array}{ccc}
            X_\iota &  & 0  \\
            &  &   \\
            0 &  & -Y_\iota  
        \end{array}
    \right]
    \leq
    \overline{L} \left[
        \begin{array}{ccc}
            Z &  & -Z  \\
            &  &   \\
            -Z &  & Z  
        \end{array}
    \right]
    +
        (2\varrho + \iota)I_{2n},
\end{equation}
and
$$
    F(x_0,X_\iota) \geq f(x_0,\xi_1) \quad\text{and}\quad F(y_0, Y_\iota) \leq f(y_0, \xi_2),
$$
where
$$
    \xi_1 = D_x \varphi(x_0,y_0) \quad \mbox{and} \quad \xi_2 = -D_y \varphi(x_0,y_0).
$$
Letting $\delta = |x_0-y_0|$, from \eqref{loc1}, we can choose $\varrho$ large enough depending on $\|u\|_\infty$ and $\gamma_0$ such that 
$$
    \delta \leq \left(\frac{1}{2}\right)^{\frac{1}{1-\gamma_0}}.
$$
As a consequence, we obtain $\omega'(\delta) \geq 1/2$, and so for $\overline{L}$ large enough depending on $\varrho$ we have 
$$
   2\overline{L} \geq |\xi_i| \geq \frac{\overline{L}}{4}.
$$
Using the equation and \eqref{RHS_growth} we have
\begin{equation}\label{equation at maximum points}
    F(x_0, X_\iota) - F(y_0, Y_\iota) \geq f(x_0,\xi_1) - f(y_0,\xi_2) \geq -8 \, \,\|q\|_\infty \overline{L}^\gamma.
\end{equation}
Notice that
\begin{eqnarray*}
    F(x_0, X_\iota) - F(y_0, Y_\iota) &=& [F(x_0, X_\iota) - F(y_0, X_\iota)] + [F(y_0, X_\iota) - F(y_0, Y_\iota)]\\
    & = & I + II
\end{eqnarray*}
Assumptions  \ref{unif ellipticity} and \ref{holder continuity of the coeffs} leads to
$$
    I + II \leq \delta^\tau (1 + \|X_\iota\|) + \mathcal{M}^+_{\lambda,\Lambda}(X_\iota - Y_\iota).
$$
Notice that by definition of $\omega$ if
$$
    r \leq \left(\frac{2-\gamma_0}{2}\right)^{\frac{1}{1-\gamma_0}}
$$
we have
$$
    \omega(r) \geq \frac{1}{2}r,
$$
which implies that
\begin{equation}\label{ineq for delta}
    \delta \leq 2 \omega(\delta) \leq \frac{4 \|u\|_\infty}{\overline{L}}.
\end{equation}
Applying inequality (\ref{VISC_INEQ_MATRICES}) for vectors of the form $(\xi,\xi)$, we obtain
$$
    (X_\iota - Y_\iota)\xi \cdot \xi \leq (4\varrho + 2\iota)|\xi|^2.
$$
And so $spec[X_\iota - Y_\iota] \subset (-\infty, 4\varrho + 2\iota]$. Now, applying the same inequality to the particular vector 
$$
    \hat{\eta}:= \frac{x_0 - y_0}{|x_0 - y_0|},
$$
we obtain
$$
    (X_\iota - Y_\iota)\hat{\eta} \cdot \hat{\eta} \leq 4Z\hat{\eta}\cdot \hat{\eta} + 4\varrho + 2\iota,
$$
where
$$
    Z =  \omega''(\delta)\,\hat{\eta} \otimes \hat{\eta} +  \frac{\omega'(\delta)}{\delta}\left(I_n - \hat{\eta} \otimes \hat{\eta} \right).
$$
Therefore,
$$
    (X_\iota - Y_\iota)\hat{\eta} \cdot \hat{\eta} \leq 4\overline{L}\,\omega''(\delta) + 4\varrho + 2\iota = - 4(1-\gamma_0)\delta^{-\gamma_0} \overline{L} + 4\varrho + 2\iota,
$$
which is a negative number. This implies that $(X_\iota-Y_\iota)$ has at least one negative eigenvalue and so
$$
    \mathcal{M}^+_{\lambda,\Lambda}(X_\iota - Y_\iota) \leq \Lambda (n-1)(8\varrho + 4\iota) - 4\lambda(1-\gamma_0)\delta^{-\gamma_0}\overline{L}.
$$
From \eqref{VISC_INEQ_MATRICES}, we obtain that
$$
    X_\iota \xi \cdot \xi \leq \overline{L}Z\xi \cdot \xi + (2\varrho + \iota)|\xi|^2 \leq \left( \overline{L} \frac{w'(\delta)}{\delta} + 2\varrho + \iota \right) |\xi|^2,
$$
and so we get an estimate from above to the positive eigenvalues of $X_\iota$, $e_i(X_\iota)^+$, leading to
$$
    0 \leq e_i(X_\iota)^+ \leq \left( \overline{L} \delta^{-1} + 2\varrho + \iota \right).
$$
To get an estimate of the negative eigenvalues, $e_i(X_\iota)^-$, we use the equation and uniform ellipticity to obtain
$$
    0 \leq -e_i(X_\iota)^- \leq \frac{1}{\lambda} \left( n\Lambda\left( \overline{L} \delta^{-1} + 2\varrho + \iota \right) + 4\|f\|_\infty \overline{L}^\gamma + |F(0,x_0)|  \right).
$$
Hence
$$
    \|X_\iota \| \leq \overline{L} \delta^{-1} + 2\varrho + \iota + \frac{1}{\lambda} \left( n\Lambda\left( \overline{L} \delta^{-1} + 2\varrho + \iota \right) + 4\|f\|_\infty \overline{L}^\gamma + |F(0,x_0)|  \right).
$$
Since $\iota$ is small, we obtain
$$
    \|X_\iota\| \leq C_0 (\overline{L}\delta^{-1} + \overline{L}^\gamma),
$$
where $C_0 = C_0(n,\lambda,\Lambda, |F(0,0)|, \|f\|_\infty)$. This implies that
$$
   I + II \leq C_0 \delta^\tau (\overline{L}\delta^{-1} + \overline{L}^\gamma) + C_1 - 4\lambda(1-\gamma_0)\delta^{-\gamma_0}\overline{L},
$$
where $C_1 = C_1(n,\Lambda,\|u\|_\infty)$. By \eqref{equation at maximum points}, we obtain
$$
    -8\, \|q\|_\infty \overline{L}^\gamma \leq C_0 \delta^\tau (\overline{L}\delta^{-1} + \overline{L}^\gamma) + C_1 - 4\lambda(1-\gamma_0)\delta^{-\gamma_0}\overline{L}.
$$
Therefore,
\begin{equation}\label{ineq for L}
    (1-\gamma_0)\delta^{-\gamma_0}\overline{L} \leq C_2 (\overline{L}\delta^{\tau-1} + \overline{L}^\gamma),
\end{equation}
where $C_2 = C_2(n,\lambda,\Lambda, \|u\|_\infty,\|q\|_\infty,|F(0,0)|)$. Note that if 
$$
    \delta \leq \left(\frac{1-\gamma_0}{2C_2} \right)^{\frac{1}{\tau - 1 + \gamma_0}},
$$
we have
$$
    ((1-\gamma_0) \delta^{-\gamma_0} - C_2 \delta^{\tau-1}) \geq \frac{(1-\gamma_0)}{2} \delta^{-\gamma_0}.
$$
This implies, by \eqref{ineq for L}, that
$$
    \frac{(1-\gamma_0)}{2} \delta^{-\gamma_0} \overline{L} \leq C_2 \overline{L}^\gamma.
$$
By \eqref{ineq for delta}, we know that
$$
    \delta^{-\gamma_0} \geq \overline{L}^{\gamma_0} (4 \|u\|_\infty)^{-\gamma_0} \geq \overline{L}^{-\gamma_0}(4\|u\|_\infty + 1)^{-1},
$$
and so 
$$
    \overline{L}^{\gamma_0+1} \leq C_2 \left(\frac{4\|u\|_\infty +1}{1-\gamma_0}\right) \overline{L}^{\gamma} = C_3 \overline{L}^\gamma.
$$
Finally
$$
    \overline{L}^{1 + \gamma_0 - \gamma} \leq C_3.
$$
and then if $\overline{L}$ sufficiently large, we get a contradiction.
\end{proof}

It is interesting to point out that the structural conditions for the Lipschitz estimates to hold is $\gamma < 2$.

Notice that once Lipschitz estimates are available, then solutions are entitled to the regularity results from \cite{Teix1}, and therefore, one can obtain up to local $C^{1,\text{Log-Lip}}$ estimates.

\section{Appendix B. Gradient growth estimates}

We dedicate this appendix to prove gradient growth estimates. Those estimates are of key importance in order to successfully execute the asymptotic analysis procedure. First, we prove its $C^{1,\alpha}$ version.

\begin{lemma}\label{growth implies grad est}
Let $u \in C(\overline{B}_1)$ be a viscosity solution to \eqref{maineq} in $B_1$ and assume, for some $\alpha \in (0,1)$, that
$$
    \sup_{x \in B_t(x_0)}\left\{|u(x)|, t|Du(x)| \right\} \leq C't^{1+\alpha}.
$$
If $u$ satisfies
$$
    \sup_{x \in B_t(x_0)}|u(x)| \leq C_0 t^{1+\alpha_1} \quad \mbox{for} \quad \alpha_1 < \left(m+1 -\frac{n}{p}\right) + (m+\gamma)\alpha,
$$
then
$$
    \sup_{x \in B_t(x_0)}|Du(x)| \leq C_0't^{\alpha_1},
$$
for $x_0 \in \mathcal{C}(u)$.
\end{lemma}
\begin{proof}
Assume $x_0 = 0$. Define
$$
    v(x) \coloneqq \frac{u(tx)}{C_0 t^{1+\alpha_1}}.
$$
Observe that $v$ is a normalized solution to
$$
    \overline{F}(x,D^2v) = \overline{f}(x,v,Dv),
$$
where
\begin{eqnarray*}
    \overline{F}(z,M) &=& \frac{t^{1-\alpha_1}}{C_0} F(t\,z, \frac{C_0}{t^{1-\alpha_1}}M)\\
    \overline{f}(z,s,\xi) & = & \frac{t^{1-\alpha_1}}{C_0}f(tz, C_0t^{1+\alpha_1}s, C_0t^{\alpha_1}\xi).
\end{eqnarray*}
Observe that, by Assumption \ref{convexity},
\begin{eqnarray*}
    |\overline{f}(x,v,Dv)| & = & \frac{t^{1-\alpha_1}}{C_0}|f(tx, u(tx), Du(tx))|\\
    & \leq & \frac{t^{1-\alpha_1}}{C_0} q(tx)\, \left|u(tx) \right|^m\, |Du(tx)|^\gamma\\
    & \leq & \frac{C'^{m+\gamma}}{C_0}q(tx)\,t^{1-\alpha_1 + m(1+\alpha_1)+\gamma \alpha},
\end{eqnarray*}
and so
$$
    \|\overline{f}\|_{L^p} \leq \frac{C'^{m+\gamma}}{C_0} \|q\|_{L^p} \,t^{1-\alpha_1 + m(1+\alpha_1)+\gamma \alpha - \frac{n}{p}}.
$$
Notice that
$$
    1-\alpha_1 + m(1+\alpha_1)+\gamma \alpha - \frac{n}{p} > 0 \iff \left(m+1 - \frac{n}{p} \right) +\gamma \alpha > \alpha_1(1-m), 
$$
which is true by assumption on $\alpha_1$. As a consequence,
$$
    \|\overline{f}\|_{L^p} \leq \frac{C'^{m+\gamma}}{C_0} \|q\|_{L^p},
$$
and so, by \cite{Teix1}, it holds
$$
    \|v\|_{C^{1,\alpha_p}(B_{1/2})} \leq L,
$$
for some universal constant $L$ and  
$$
    \alpha_p = \min \left\{1 - \frac{n}{p}, \alpha_*^- \right\}.
$$
In particular,
$$
    |Dv(x)| \leq L,
$$
for $x \in B_{1/2}$, which is equivalently to
$$
    |Du(x)| \leq C_0 L t^{\alpha_1}.
$$
\end{proof}

We also deploy its second-order version.

\begin{lemma}\label{growth est sec-prd version}
Let $u \in C(\overline{B}_1)$ be a viscosity solution to \eqref{maineq} in $B_1$ and assume, for some $\alpha \in (0,1)$, that
$$
    \sup_{x \in B_t(x_0)}\left\{|u(x)|, t|Du(x)| \right\} \leq C't^{2+\alpha}.
$$
If $u$ satisfies
$$
    \sup_{x \in B_t(x_0)}|u(x)| \leq C_0 t^{2+\alpha_1} \quad \mbox{for} \quad \alpha_1 < \left(2m+\gamma -\frac{n}{p}\right) + \gamma\,\alpha,
$$
then
$$
    \sup_{x \in B_t(x_0)}|Du(x)| \leq C_0't^{1+\alpha_1},
$$
for $x_0 \in \mathcal{C}(u)$.
\end{lemma}
\begin{proof}
The proof follows the same lines as the proof of Lemma \ref{growth implies grad est}. We can assume $x_0 = 0$. Define
$$
    v(x) \coloneqq \frac{u(tx)}{C_0 t^{2+\alpha_1}}.
$$
Observe that $v$ is a normalized solution to
$$
    \overline{F}(x,D^2v) = \overline{f}(x,v,Dv),
$$
where
\begin{eqnarray*}
    \overline{F}(z,M) &=& \frac{t^{-\alpha_1}}{C_0} F(t\,z, \frac{C_0}{t^{-\alpha_1}}M)\\
    \overline{f}(z,s,\xi) & = & \frac{t^{-\alpha_1}}{C_0}f(tz, C_0t^{2+\alpha_1}s, C_0t^{1+\alpha_1}\xi).
\end{eqnarray*}
As before, we can estimate
$$
    \|\overline{f}\|_{L^p} \leq \frac{C'^{m+\gamma}}{C_0} \|q\|_{L^p}t^{m(2+\alpha_1) + (1+\alpha)\gamma - \alpha_1 - \frac{n}{p}},
$$
and repeat the same arguments as in the proof of Lemma \ref{growth implies grad est} by noticing that
$$
    m(2+\alpha_1) + (1+\alpha)\gamma - \alpha_1 - \frac{n}{p} > 0 \iff \left(2m + \gamma - \frac{n}{p} \right) + \alpha \gamma > \alpha_1 (1-m),
$$
which is true by assumption on $\alpha_1$.
\end{proof}
\bibliographystyle{amsplain, amsalpha}

\end{document}